\documentclass[12pt, oneside,usenames,dvipsnames]{amsart}

\usepackage{amsmath,ifthen, amsfonts, amssymb, 
srcltx, 
   amsopn,tikz, tkz-euclide, url}
\usetikzlibrary{decorations.markings,arrows, cd}
\usepackage{tikz-cd}
 \usepackage{dutchcal}
\usepackage{xfrac} 
\usepackage{graphicx, ifsym}
\usepackage{fullpage}
\tikzset{>=latex}

\usepackage{thmtools}
\usepackage{thm-restate}
\usepackage{pgfplots}
\usetikzlibrary{intersections, pgfplots.fillbetween}

\usepackage{hyperref}

\usepackage{cleveref}

\newcommand{\showcomments}{yes}
\renewcommand{\showcomments}{no}

\newsavebox{\commentbox}
\newenvironment{com}%
{\ifthenelse{\equal{\showcomments}{yes}}%
{\footnotemark
        \begin{lrbox}{\commentbox}
        \begin{minipage}[t]{1.25in}\raggedright\sffamily\tiny
        \footnotemark[\arabic{footnote}]}
{\begin{lrbox}{\commentbox}}}%
{\ifthenelse{\equal{\showcomments}{yes}}%
{\end{minipage}\end{lrbox}\marginpar{\usebox{\commentbox}}}
{\end{lrbox}}}

\usepackage{amsmath,amsfonts, amssymb, graphicx, tikz, bbding,ifthen,srcltx,amsopn, url}
\usepackage{bm}
\usetikzlibrary{decorations.markings,arrows, intersections}
\usetikzlibrary{decorations.pathreplacing}

\DeclareMathOperator{\rank}{rk}

\theoremstyle{definition}
\newtheorem{thm}{Theorem}[section]
\newtheorem{lem}[thm]{Lemma}
\newtheorem{prop}[thm]{Proposition}

\newtheorem{exa}[thm]{Example}

\newtheorem{question}[thm]{Question}
\newtheorem{remark}[thm]{Remark}

\newtheorem{cor}[thm]{Corollary}

\newtheorem{thmx}{Theorem}
\newtheorem{corx}[thmx]{Corollary}

\author{Kasia Jankiewicz}
\address{Department of Mathematics, University of Chicago, Chicago, Illinois, 60637}
\email{kasia@math.uchicago.edu}
\title{Splittings of triangle Artin groups}

\begin{document}
\begin{com}
{\bf \normalsize COMMENTS\\}
ARE\\
SHOWING!\\
\end{com}

\begin{abstract} We show that a triangle Artin group $\text{Art}_{MNP}$ where $M\leq N\leq P$ splits as an amalgamated product or an HNN extension of finite rank free groups, provided that either $M>2$, or $N>3$. We also prove that all even three generator Artin groups are residually finite.
\end{abstract}
\maketitle


A triangle Artin group is given by the presentation
$$\text{Art}_{MNP} = \langle a, b, c\mid (a, b)_{M} = (b, a)_{M}, (b,c)_{N} = (c,b)_{N}, (c,a)_{P} = (a,c)_{P}\rangle,$$
where $(a,b)_M$ denote the alternating word $aba\dots$ of length $M$. 
Squier showed that the Euclidean triangle Artin group, i.e.\ $\text{Art}_{236}, \text{Art}_{244}$ and $\text{Art}_{333}$, split as amalgamated products or an HNN extension of finite rank free groups along finite index subgroups \cite{Squier87}. We generalize that result to other triangle Artin groups.
\begin{restatable}{thmx}{splitting}
\label{thm:splitting} 
Suppose that $M\leq N\leq P$ where either $M>2$, or $N>3$. Then the Artin group $\text{Art}_{MNP}$ splits as an amalgamated product or an HNN extension of finite rank free groups.
\end{restatable}

The assumptions of the above theorem are satisfied for all triples of numbers except for $(2,2,P)$ and $(2,3,P)$. 
An Artin group is \emph{spherical}, if the associated Coxeter group is finite. A three generator Artin group $\text{Art}_{MNP}$ is spherical exactly when $\frac{1}{M}+\frac{1}{N}+\frac{1}{P}>1$, i.e.\ $(M,N,P) = (2,2,P)$ or $(2,3,3), (2,3,4), (2,3,5)$. 
All spherical Artin groups have infinite center, and none of them splits as a graph of finite rank free groups (see \Cref{prop:cyclic center}). The remaining cases are $(2,3,P)$ where $P\geq 6$. The above theorem holds for triple $(2,3,6)$ by \cite{Squier87}.  It remains unknown for $(2,3,P)$ with $P\geq 7$.
The cases where $M>2$ were considered in \cite[Thm B]{JankiewiczArtinRf} and it was proven that they all split as amalgamated products of finite rank free groups. 

Graphs of free groups form an important family of examples in geometric group theory. Graph of free groups with cyclic edge groups that contain no Baumslag-Solitar subgroups are virtually special \cite{HsuWiseCubulating}, and contain quasiconvex surface subgroups \cite{Wilton18}. Graphs of free groups with arbitrary edge groups can exhibit various behaviors.
For example, an amalgamated product $A*_CB $ of finite rank free groups  where $C$ is malnormal in $A,B$
is hyperbolic \cite{BestvinaFeighn92}, and virtually special \cite{HsuWiseCubulatingMalnormal}. 
On the other hand there are examples of amalgamated products of finite rank free groups that are not residually finite \cite{Bhattacharjee94},  \cite{Wise96Thesis}, and even simple \cite{BurgerMozes97}. The last two arise as lattices in the automorphism group of a product of two trees.

By further analysis of the splitting, we are also able to show that some of the considered Artin groups are residually finite.
\begin{thmx}\label{thm:rf}
The Artin group $\text{Art}_{2MN}$ where $M,N\geq 4$ and at least one of $M,N$ is even, is residually finite.
\end{thmx}
An Artin group $\text{Art}_{MNP}$ is \emph{even} if all $M,N,P$ are even.
The above theorem combined with our result in \cite{JankiewiczArtinRf} (and the fact that $\text{Art}_{22P} = \mathbb Z\times \text{Art}_P$ is linear) gives us the following. 
\begin{corx}
All even Artin groups on three generators are residually finite. 
\end{corx}

All linear groups are residually finite \cite{Malcev40}, so residual finiteness can be viewed as testing for linearity.
Spherical Artin groups are known to be linear (\cite{Krammer2002}, \cite{Bigelow2001} for braid groups, and \cite{CohenWales2002}, \cite{Digne2003} for other spherical Artin groups).
The right-angled Artin groups are also well known to be linear, but not much more is known about linearity of Artin groups. 
In last years, a successful approach in proving that groups are linear is by showing that they are virtually special. Artin groups whose defining graphs are forests are the fundamental groups of graph manifolds with boundary \cite{Brunner92}, \cite{HermillerMeier99}, 
and so they are virtually special \cite{LiuGraphManifolds}, \cite{PrzytyckiWiseGraphManifolds}. 
Many Artin groups in certain classes (including $2$-dimensional, or three generator) 
are not cocompactly cubulated even virtually, unless they are sufficiently similar to RAAGs 
\cite{HuangJankiewiczPrzytycki16}, \cite{HaettelArtin}. In particular, the only (virtually) cocompactly cubulated three generator Artin group are $\text{Art}_{22M} = \mathbb Z \times \text{Art}_{M}$, $\text{Art}_{MN\infty}$ where $M,N$ are both even, $\text{Art}_{M\infty\infty} = \mathbb Z*\text{Art}_M$, and $\text{Art}_{\infty\infty\infty} = F_3$. 
Some triangle-free Artin groups 
act properly but not cocompactly on locally finite, finite dimensional 
CAT(0) cube complexes \cite{HaettelTriangleFreeArtin}.

In \cite{JankiewiczArtinRf} we showed that $\text{Art}_{MNP}$ are residually finite when $M,N,P\geq 3$, except for the cases where $(M,N,P) = (3,3,2p+1)$ with $p\geq 2$. Few more families of Artin groups are known to be residually finite, e.g.\ even FC type Artin groups~\cite{BlascoGarciaMartinezPerezParis19}, and certain triangle-free Artin groups \cite{BlascoGarciaJuhaszParis18}.

\subsection*{Organization}
In Section~\ref{sec:background} we provide some background. In Section~\ref{sec:splitting} we prove Theorem~\ref{thm:splitting} as Proposition~\ref{prop:all even splitting} and Corollary~\ref{cor:at least one odd splitting}. We also show that the only irreducible spherical Artin groups splitting as graph of finite rank free groups are dihedral. In Section~\ref{sec:rf} we recall a criterion for residual finiteness of amalgamated products and HNN extensions of free groups from \cite{JankiewiczArtinRf}  and prove Theorem~\ref{thm:rf}.

\subsection*{Acknowledgements}
This material is based upon work supported by the National Science Foundation
under Grant No.\ DMS-1928930 while the author participated in a program hosted
by the Mathematical Sciences Research Institute in Berkeley, California, during 
the Fall 2020 semester.
\section{Background}\label{sec:background}

\subsection{Graphs}
Let $X$ be a finite graph with directed edges. We denote the vertex set of $X$ by $V(X)$ and the edge set of $X$ by $E(X)$. The vertices incident to an edge $e$ are denoted by $e^+$ and $e^-$. 
A \emph{map of graphs} $f:X_1\to X_2$ sends vertices to vertices, and edges to concatenations of edges.
A map $f$ is a \emph{combinatorial map} if single edges are mapped to single edges. A combinatorial map $f$ is a \emph{combinatorial immersion} if given two edges $e_1, e_2$ such that $e_1^- = e_2^-$ we have $f(e_1) = f(e_2)$ (as oriented edges) if and only if $e_1=e_2$.
Consider two edges $e_1, e_2$ with $e_1^- = e_2^-$. A \emph{fold} is the natural combinatorial map $X\to \bar X$ where $V(\bar X) = V(X)/{e_1^+\sim e_2^+}$ and $E(\bar X)/e_1\sim e_2$. Stallings showed that every combinatorial map $X\to X'$ factors as $X \to \bar X \to X'$ where $X\to \bar X$ is a composition of finitely many folds, and $\bar X \to X'$ is a combinatorial immersion \cite{Stallings83}. We refer to $X\to \bar X$ as a \emph{folding map}.

\subsection{Maps between free groups}
Let $H,G$ be finite rank free groups. 
Let $Y$ be a bouquet of $n=\rank G$ circles. 
We can identify $\pi_1 Y \simeq F_n$ with $G$ by orienting and labelling edges of $Y$ with the generators of $G$. 
Every homomorphism $\phi:H\to G$ can be represented by a combinatorial immersion of graphs. 
Indeed, start with a map of graphs $X\to Y$ where $X$ is a bouquet of $m=\rank H$ circles. 
We think of each circle in $X$ as subdivided with edges oriented and labelled by the generators of $G$, 
so that each circles is labelled by a word from a generating set of $H$. 
By Stallings, the map $X\to Y$ factors as $X\to \bar X\to Y$ where $X\to \bar X$ is a folding map, 
and $\bar X \to Y$ is a combinatorial immersion. 
Indeed, $\bar X$ is obtained by identifying two edges with the same orientation and label that share an endpoint. 

Note that the rank of $\phi(H)$ is equal $\rank \pi_1\bar X = 1-\chi (\bar X)$ where $\chi$ denotes the Euler characteristic. In particular, a homomorphism $\phi$ is injective if and only if the folding map $X\to \bar X$ is a homotopy equivalence.  In that case, $\bar X$ is a precover of $Y$ which can be completed to a cover of $Y$ corresponding to the subgroup $H$ of $G$ via the Galois correspondence.
In particular, every subgroup of $G$ is uniquely represented by a combinatorial immersion $(X,x)\to (Y,y)$ where $y$ is the unique vertex of $Y$, and $X$ is a folded graph with basepoint $x$.
We refer to \cite{Stallings83} for more details.

\subsection{Intersections of subgroups of a free group}
Let $Y$ be a graph, and $\rho_i:(X_i,x_i)\to (Y,y)$ a combinatorial immersion for $i=1,2$. 
The \emph{fiber product of $X_1$ and $X_2$ over $Y$}, denoted $X_1\otimes_Y X_2$ is a graph with the vertex set 
$$V(X_1\otimes_Y X_2) = \{(v_1,v_2)\in V(X_1)\times V(X_2): \rho_1(v_1) = \rho_2(v_2)\},$$
and the edge set
$$E(X_1\otimes_Y X_2) = \{(e_1, e_2)\in E(X_1)\times E(X_2): \rho_1(e_1) = \rho_2(e_2)\}.
$$
The graph $X_1\otimes_Y X_2$ often has several connected components. 
There is a natural combinatorial immersion $X_1\otimes_Y X_2\to Y$, and it induces an embedding $\pi_1(X_1\otimes_Y X_2, (x_1, x_2))\to \pi_1(Y,y)$. We have the following.

\begin{thm}[ {\cite[Thm 5.5]{Stallings83}}]\label{thm:fiber product}
Let $H_1, H_2$ be two subgroups of $G = \pi_1Y$, and for $i=1,2$ let $(X_i, x_i)\to (Y,y)$ be a combinatorial immersion of graphs inducing the inclusion $H_i\hookrightarrow G$.
The intersection $H_1\cap H_2$ is represented by a combinatorial immersion $\left(X_1\otimes_Y X_2, (x_1, x_2)\right)\to (Y,y)$.
\end{thm}

In particular, when $Y$ is a bouquet of circles with $\pi_1Y = G$, 
and $(X,x)\to (Y,y)$ is a combinatorial immersion inducing $H = \pi_1X\hookrightarrow G$, 
then for every pair of (not necessarily distinct) vertices $x_1, x_2\in X$, 
the group $\pi_1\left(X\otimes_Y X, (x_1, x_2)\right)$ is an intersection $H^{g_1}\cap H^{g_2}$ for some $g_1, g_2 \in G$. 
In fact, every non-trivial intersection $H\cap H^g$ is equal $\pi_1\left(X\otimes_Y X, (x_1, x_2)\right)$ 
where $x_1 = x$, and $x_2$ is some (possibly the same) vertex in $X$. 
The connected component of $X\otimes_Y X$ containing $(x, x)$ is a copy of $X$, 
which we refer to as a \emph{diagonal component}. 
The group $\pi_1\left(X\otimes_Y X, (x, x)\right)$ is the intersection $H\cap H^g = H$, i.e.\ where $g\in H$. 
A connected component of $X\otimes_Y X$ that has no edges is called \emph{trivial}.
\subsection{Graph of groups and spaces}
We recall the definitions of a graph of groups and a graph of spaces, following \cite{ScottWall79}.

A \emph{graph of spaces} consists of 
\begin{itemize}
\item a graph $\Gamma$, called the \emph{underlying graph},
\item a collection of CW-complexes $X_{v}$ for each $v\in V(\Gamma)$, called \emph{vertex spaces}, 
\item a collection of CW-complexes $X_{e}$ for each $e\in E(\Gamma)$, called \emph{edge spaces}, and
\item a collection of continuous $\pi_1$-injective maps $f_{(e,\pm)}:X_e\to X_{e^{\pm}}$ for each $e\in E(\Gamma)$.
\end{itemize}
The \emph{total space} $X(\Gamma)$ is defined as 
$$X(\Gamma) = \bigsqcup_{v\in V(\Gamma)} X_v \sqcup \bigsqcup_{e\in E(\Gamma)} X_e \times [-1,1]/\sim$$
where $(x,\pm 1)\sim f_{(e, \pm)}(x)$ for $x\in X_e$.

Similarly, a \emph{graph of groups} consists of 
\begin{itemize}
\item the \emph{underlying graph} $\Gamma$,
\item a collection of \emph{vertex groups} $G_{v}$ for each $v\in V(\Gamma)$,
\item a collection of \emph{edge groups} $G_{e}$ for each $e\in E(\Gamma)$, and
\item a collection of injective homomorphisms $\phi_{(e,\pm)}:G_e\to G_{e^{\pm}}$ for each $e\in E(\Gamma)$.
\end{itemize}
The \emph{fundamental group of a graph of groups} is defined as the fundamental group of the graph of spaces $X(\Gamma)$, where $X_v = K(G_v, 1)$ for each $v\in V(\Gamma)$, $X_e=K(G_e, 1)$ for each edge $e\in E(\Gamma)$, and $f_{(e,\pm)}$ induces homomorphism $\phi_{(e,\pm)}$ on the fundamental groups. Note that the fundamental group $\pi_1 \Gamma$ is a subgroup of $G$. 

\subsection{HNN extensions and doubles}
We will denote the \emph{HNN extension of $A$ relative to $\beta:B\to A$} where $B\subseteq A$ by $A*_{B,\beta}$, i.e.\
\[A*_{B,\beta} = \langle A, t\mid t^{-1}xt = \beta(x) \text{ for all }x\in B\rangle.\]
The generator $t$ is called the \emph{stable letter}.
 Note that $A*_{B,\beta}$ can be viewed as a graph of group $G(\Gamma)$ where $\Gamma$ is a single vertex $v$ with a single loop $e$, $G_v = A$, $G_e = B$, $\phi_{(e,-)}$ is the inclusion of $B$ in $A$, and $\phi_{(e,+)}=\beta$. 
 
 A \emph{double of $A$ along $C$ twisted by an automorphism $\beta:C\to C$}, denoted by $D(A,C,\beta)$, 
is an amalgamated product $A*_CA$, where $C$ is mapped to the first factor via the standard inclusion, 
and to the second via the standard inclusion precomposed with $\beta$. 
As usual, $D(A,C,\beta)$ depends only on the outer automorphism class of $\beta$, and not a particular representative.
A double $D(A,C,\beta)$ can be viewed as a graph of groups $G(\Gamma)$ 
where $\Gamma$ is a single edge $e$ with distinct endpoints, $G_{e^{\pm}} = A$, $G_e = C$, 
and $\phi_{(e,-)}$ is the inclusion of $C$ in $A$, and $\phi_{(e,+)}$ is the inclusion precomposed with $\beta$. 
Note that an amalgamated product $A*_CB$ where $[B:C]=2$ has an index two subgroup $D(A,C,\beta)$ 
where $\beta:C\to C$ is conjugation by some (any) representative $g\in B$ of the non-trivial coset of $B/C$.

In both situtaions where there is unique edge in the underlying graph $\Gamma$, we will skip the label $e$, and we will denote $\phi_{(e, \pm)}$ simply by $\phi_{\pm}$. Similarly, in a graph of spaces with a unique edge $e$, we will write $f_{\pm}$ instead of $f_{(e,\pm)}$.

\subsection{Triangle groups}
A \emph{triangle (Coxeter) group} is given by the presentation
\[
W_{MNP} = \langle a,b, c\mid a^2, b^2, c^2, (ab)^M, (bc)^N, (ca)^P\rangle.
\]
The group $W_{MNP}$ acts as a reflection group on
\begin{itemize}
\item the sphere, if $\frac 1M +\frac 1N +\frac 1p >1$,
\item the Euclidean plane,  if $\frac 1M +\frac 1N +\frac 1p =1$,
\item the hyperbolic plane, if $\frac 1M +\frac 1N +\frac 1p <1$.
\end{itemize} 
Hyperbolic triangle groups are commensurable with the fundamental groups of negatively curved surfaces, and therefore they are locally quasiconvex, virtually special, and Gromov-hyperbolic. 
A \emph{von Dyck triangle group} is an index $2$ subgroup of $W_{MNP}$ with the presentation
\[\langle x,y\mid x^M, y^N, (x^{-1}y)^P\rangle\]
obtained by setting $x = ba$ and $y=bc$.

\section{Splittings}\label{sec:splitting}

The goal of this section is to prove \Cref{thm:splitting}.
In \cite{JankiewiczArtinRf} we proved the following.
\begin{thm}[{\cite[Thm B]{JankiewiczArtinRf}}]
The Artin group $\text{Art}_{MNP}$ with $M,N,P\geq 3$ splits as an amalgamated product or an HNN extension of finite rank free groups.
\end{thm}

The remaining cases in \Cref{thm:splitting} are $\text{Art}_{2MN}$ where $M,N\geq 4$. The case where $M,N$ are both even is \Cref{prop:all even splitting}, and the other case is \Cref{cor:at least one odd splitting}. We start with considering a non-standard presentation of $\text{Art}_{2MN}$. In the next two subsections we construct the splittings. Then we show that the only spherical Artin groups that split as graphs of free groups are dihedral Artin groups. Spherical Artin group on $3$ generators $\text{Art}_{MNP}$ correspond to one of the following triples $(2,2,P)$ for any $P\geq 2$, $(2,3,3), (2,3,4)$ and $(2,3,5)$. For completeness, in the last subsection we include the proof that the three generator Artin groups with at least one $\infty$ label admit splittings as HNN extensions or amalgamated products of finite rank free groups.

The Artin group $\text{Art}_{236}$ splits as $F_3*_{F_7}F_4$ by Squier \cite{Squier87}.
The only remaining three generator Artin groups are $\text{Art}_{23M}$ where $M\geq 7$, and the following remains unanswered.
\begin{question}
Does the Artin group $\text{Art}_{23M}$ where $M\geq 7$ splits as a graph of finite rank free groups?
\end{question}
We conjecture that the answer is positive. More generally, we ask the following.
\begin{question}
Do all $2$-dimensional Artin group split as a graph of finite rank free groups?
\end{question}

\subsection{Presentations of $\text{Art}_{2MN}$}
Here is the standard presentation of Artin group $\text{Art}_{2MN}$:

$$\langle a,b,c \mid (a,b)_M = (b,a)_M, (b,c)_N = (c,b)_N, ac=ca \rangle.$$

Let $x = ab$ and $y=cb$ and consider a new presentation of $\text{Art}_{2MN}$ with generators $b,x,y$. 
The relation $(a,b)_{M} = (b,a)_M$ is replaced by $bx^mb^{-1} = x^m$ when $M=2m$, 
and by $bx^mb = x^{m+1}$ when $M=2m+1$.
We denote this relation by $r_{M}(b,x)$.
Note that $yx^{-1} = ca^{-1}$, so relation $ac = ca$ can be replaced by $yx^{-1} = bx^{-1}yb^{-1}$. 
See Figure~\ref{fig:horizontal graphs}(a).

This gives us the following presentation 
\begin{equation}\label{eq:2MN presentation}\tag{$*$}\text{Art}_{2MN} = \langle b,x,y \mid r_{M}(b,x), r_{N}(b,y), bx^{-1}yb^{-1}=yx^{-1}\rangle.\end{equation}

Let $X_{2MN}$ be the presentation complex associated to the presentation~\eqref{eq:2MN presentation}. 
Let $X_A$ be the bouquet of two loops labelled by $x$ and $y$.
The complex $X_{2MN}$ can be viewed as a union of the graph $X_A$ and for each relation in ~\eqref{eq:2MN presentation}: a cylinder (for relations $bx^{-1}yb^{-1}=yx^{-1}$ and each $r_{M}(b,x)$ with $M$ even) or a Mobius strip (for each relation $r_{M}(b,x)$ with $M$ odd) with boundary cycles are glued to $X_A$. We can metrize them so that the height of each cylinder/Mobius strip is equal $2$. 

We now define a map $p:X_{2MN}\to [0,1]$ by describing the restriction of $p$ to each cylinder/Mobius strip of $X_{2MN}$. Each point of the cylinder/Mobius strip is mapped to its distance from the center circle of that cylinder/Mobius strip. In particular, the center circle of each cylinder or Mobius strip is mapped to $0$, and the boundary circles of the cylinder or Mobius strip are mapped to $1$. See Figure~\ref{fig:horizontal graphs}(a). 

We can identify $X_A $ with the preimage $p^{-1}(1)$. We define a graph $X_B$ as the union of all the center circles, i.e.\ the preimage $p^{-1}(0)$. We also define a subgraph $X_C$ of $X_{2MN}$ as the preimages $p^{-1}(\frac 12)$. The graphs $X_A, X_B$ and $X_C$ are illustrated in Figure~\ref{fig:horizontal graphs}(b).  The graph $X_C$ has two vertices, which are its intersections with the edge $b$. We denote them by $b_-, b_+$, so that $b_-$, the midpoint of the edge $b$, $b_+$ are ordered consistently with the orientation of the edge $b$.
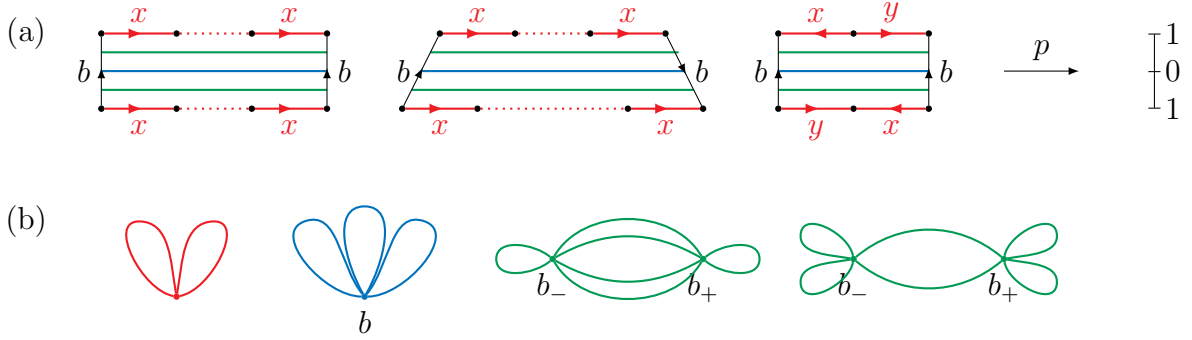
\begin{figure}
\begin{tikzpicture}[decoration={
    markings,
    mark=at position 0.55 with {\arrow{>}}}]
\begin{scope}
\node[draw=none,fill=none] (x) at (-1, 1) {(a)};
\node[circle, draw, fill, inner sep = 0pt,minimum width = 2pt] (a) at (0,0) {};
\node[circle, draw, fill, inner sep = 0pt,minimum width = 2pt] (b) at (0,1) {};
\node[circle, draw, fill, inner sep = 0pt,minimum width = 2pt] (c) at (1,1) {};
\node[circle, draw, fill, inner sep = 0pt,minimum width = 2pt] (d) at (2,1) {};
\node[circle, draw, fill, inner sep = 0pt,minimum width = 2pt] (e) at (3,1) {};
\node[circle, draw, fill, inner sep = 0pt,minimum width = 2pt] (f) at (3,0) {};
\node[circle, draw, fill, inner sep = 0pt,minimum width = 2pt] (g) at (2,0) {};
\node[circle, draw, fill, inner sep = 0pt,minimum width = 2pt] (h) at (1,0) {};
\draw[RoyalBlue, thick] (0,0.5) -- (3, 0.5);
\draw[ForestGreen, thick] (0, 0.25) -- (3,0.25);
\draw[ForestGreen, thick] (0, 0.75) -- (3,0.75);
\draw[postaction={decorate}] (a) -- (b) node[draw=none,fill=none,midway,left] {$b$};
\draw[postaction={decorate}, Red, thick] (b) -- (c) node[draw=none,fill=none,midway,above] {$x$};
\draw[dotted, Red, thick] (c) -- (d) node[draw=none,fill=none,midway,right] {};
\draw[postaction={decorate}, Red, thick] (d) -- (e) node[draw=none,fill=none,midway,above] {$x$};
\draw[postaction={decorate}, Red, thick] (a) -- (h) node[draw=none,fill=none,midway,below] {$x$};
\draw[dotted, Red, thick] (h) -- (g) node[draw=none,fill=none,midway,right] {$$};
\draw[postaction={decorate}, Red, thick] (g) -- (f) node[draw=none,fill=none,midway,below] {$x$};
\draw[postaction={decorate}] (f) -- (e) node[draw=none,fill=none,midway,right] {$b$};
\end{scope}
\begin{scope}[shift={(4,0)}]
\node[circle, draw, fill, inner sep = 0pt,minimum width = 2pt] (a) at (0,0) {};
\node[circle, draw, fill, inner sep = 0pt,minimum width = 2pt] (b) at (0.5,1) {};
\node[circle, draw, fill, inner sep = 0pt,minimum width = 2pt] (c) at (1.5,1) {};
\node[circle, draw, fill, inner sep = 0pt,minimum width = 2pt] (d) at (2.5,1) {};
\node[circle, draw, fill, inner sep = 0pt,minimum width = 2pt] (e) at (3.5,1) {};
\node[circle, draw, fill, inner sep = 0pt,minimum width = 2pt] (f) at (4,0) {};
\node[circle, draw, fill, inner sep = 0pt,minimum width = 2pt] (g) at (3,0) {};
\node[circle, draw, fill, inner sep = 0pt,minimum width = 2pt] (h) at (1,0) {};
\draw[RoyalBlue, thick] (0.25,0.5) -- (3.75, 0.5);
\draw[ForestGreen, thick] (0.125, 0.25) -- (3.875,0.25);
\draw[ForestGreen, thick] (0.375, 0.75) -- (3.675,0.75);
\draw[postaction={decorate}] (a) -- (b) node[draw=none,fill=none,midway,left] {$b$};
\draw[postaction={decorate}, Red, thick] (b) -- (c) node[draw=none,fill=none,midway,above] {$x$};
\draw[dotted, Red, thick] (c) -- (d) node[draw=none,fill=none,midway,right] {};
\draw[postaction={decorate}, Red, thick] (d) -- (e) node[draw=none,fill=none,midway,above] {$x$};
\draw[postaction={decorate}, Red, thick] (a) -- (h) node[draw=none,fill=none,midway,below] {$x$};
\draw[dotted, Red, thick] (h) -- (g) node[draw=none,fill=none,midway,right] {$$};
\draw[postaction={decorate}, Red, thick] (g) -- (f) node[draw=none,fill=none,midway,below] {$x$};
\draw[postaction={decorate}] (e) -- (f) node[draw=none,fill=none,midway,right] {$b$};
\end{scope}
\begin{scope}[shift={(9,0)}]
\node[circle, draw, fill, inner sep = 0pt,minimum width = 2pt] (a) at (0,0) {};
\node[circle, draw, fill, inner sep = 0pt,minimum width = 2pt] (b) at (0,1) {};
\node[circle, draw, fill, inner sep = 0pt,minimum width = 2pt] (c) at (1,1) {};
\node[circle, draw, fill, inner sep = 0pt,minimum width = 2pt] (d) at (2,1) {};
\node[circle, draw, fill, inner sep = 0pt,minimum width = 2pt] (g) at (2,0) {};
\node[circle, draw, fill, inner sep = 0pt,minimum width = 2pt] (h) at (1,0) {};
\draw[RoyalBlue, thick] (0,0.5) -- (2, 0.5);
\draw[ForestGreen, thick] (0, 0.25) -- (2,0.25);
\draw[ForestGreen, thick] (0, 0.75) -- (2,0.75);
\draw[postaction={decorate}] (a) -- (b) node[draw=none,fill=none,midway,left] {$b$};
\draw[postaction={decorate}, Red, thick] (c) -- (b) node[draw=none,fill=none,midway,above] {$x$};
\draw[postaction={decorate}, Red, thick] (c) -- (d) node[draw=none,fill=none,midway,above] {$y$};
\draw[postaction={decorate}, Red, thick] (a) -- (h) node[draw=none,fill=none,midway,below] {$y$};
\draw[postaction={decorate}, Red, thick] (g) -- (h) node[draw=none,fill=none,midway,below] {$x$};
\draw[postaction={decorate}] (g) -- (d) node[draw=none,fill=none,midway,right] {$b$};
\end{scope}

\begin{scope}[shift={(14,0)}]
\draw[->] (-2,0.5) -- (-1,0.5) node[draw=none,fill=none,midway,above] {$p$};
\node[draw=none,fill=none,right] (a) at (0,0) {$1$};
\node[draw=none,fill=none,right] (b) at (0,0.5) {$0$};
\node[draw=none,fill=none,right] (c) at (0,1) {$1$};
\draw[|-|] (0,0) -- (0,0.5);
\draw[-|] (0,0.5) -- (0,1);
\end{scope}
\begin{scope}[shift={(1,-2.5)}]
\node[draw=none,fill=none] (x) at (-2, 1) {(b)};
\node[circle, draw, fill, inner sep = 0pt,minimum width = 2pt, Red] (a) at (0,0) {};
\draw[
Red, thick] (a) to[out = 80, in =180] (0.5,1) to[out=0, in =10] (a);
\draw[
Red, thick] (a) to[out = 100, in =0] (-0.5,1) to[out=180, in =170] (a) ;
\end{scope}
\begin{scope}[shift={(3.5,-2.5)}]
\node[circle, draw, fill, inner sep = 0pt,minimum width = 2pt, RoyalBlue, label=below:$b$] (a) at (0,0) {};
\draw[RoyalBlue, thick] (a) to[out = 120, in =180] (0,1.2) to[out=0, in =60] (a);
\draw[RoyalBlue, thick] (a) to[out = 50, in =160] (0.8,1) to[out=-20, in =0] (a);
\draw[RoyalBlue, thick] (a) to[out = 130, in =20] (-0.8,1) to[out=200, in =180] (a);
\end{scope}
\begin{scope}[shift={(6,-2)}]
\node[circle, draw, fill, inner sep = 0pt,minimum width = 2pt,ForestGreen, label=below:$b_-$] (-) at (0,0) {};
\node[circle, draw, fill, inner sep = 0pt,minimum width = 2pt,ForestGreen, label=below:$b_+$] (+) at (2,0) {};
\draw[ForestGreen, thick] (-) to[out = 30, in =150] (+);
\draw[ForestGreen, thick] (-) to[out = 60, in =120] (+);
\draw[ForestGreen, thick] (-) to[out = -30, in =210] (+);
\draw[ForestGreen, thick] (-) to[out = -60, in =240] (+);
\draw[ForestGreen, thick] (-) to[out = 140, in =90] (-0.75,0) to[out = 270, in =220] (-);
\draw[ForestGreen, thick] (+) to[out = 40, in =90] (2.75,0) to[out = 270, in =-40] (+);
\end{scope}
\begin{scope}[shift={(10,-2)}]
\node[circle, draw, fill, inner sep = 0pt,minimum width = 2pt,ForestGreen, label=below:$b_-$] (-) at (0,0) {};
\node[circle, draw, fill, inner sep = 0pt,minimum width = 2pt,ForestGreen, label=below:$b_+$] (+) at (2,0) {};
\draw[ForestGreen, thick] (-) to[out = 40, in =140] (+);
\draw[ForestGreen, thick] (-) to[out = -40, in =220] (+);
\draw[ForestGreen, thick] (+) to[out = -10, in =80] (2.7,-0.35) to[out=260, in =280] (+);
\draw[ForestGreen, thick] (+) to[out = 10, in =280] (2.7,0.35) to[out=100, in =80] (+) ;
\draw[ForestGreen, thick] (-) to[out = 190, in =100] (-0.7, -0.35) to[out=280, in =260] (-);
\draw[ForestGreen, thick] (-) to[out = 170, in =260] (-0.7,0.35) to[out=80, in =100] (-) ;
\end{scope}
\end{tikzpicture}
\caption{ (a) The relation $R_M(b,x)$ where $M$ is even (left) and odd (middle), and the relation $bx^{-1}yb^{-1}=yx^{-1}$ (right), with the projection $p$ of the cells onto the interval $[0,1]$. The horizontal graphs $X_A, X_B, X_C$ are the preimages $p^{-1}(1), p^{-1}(0), p^{-1}(\frac 12)$ respectively. (b) Graphs $X_A, X_B$ and two versions of $X_C$ depending on whether $M,N$ are both odd (left green), or one of $M,N$ is even (right green).}
\label{fig:horizontal graphs}
\end{figure}
When $M,N$ are both even, then the graph $X_C$ is not connected. Indeed, each of its connected components is a copy of $X_B$. We denote the connected component containing the vertex $b_-$ by $X_B^-$, and the component containing the vertex $b_+$ by $X_B^+$. 
Otherwise, if at least one $M,N$ is odd, then $X_C$ is a connected double cover of $X_B$. In the next two sections, we describe the graph of spaces decomposition of $X_{2MN}$ associated to the map $p$, and the induced graph of groups decomposition of $\text{Art}_{2MN}$. We consider separately the case where $M,N$ are both even, and the case where at least one of them is odd.

\subsection{Both even}
In the case where both $M,N$ are even $\geq 4$, Presentation~\eqref{eq:2MN presentation} of $\text{Art}_{2MN}$ is the standard presentation of an HNN-extension. 
\begin{prop}\label{prop:all even splitting}
Let $M = 2m$ and $N=2n$ be both even and $\geq 4$. Then $\text{Art}_{2MN}$ splits as an HNN-extension $A*_{B,\beta}$ where $A=\langle x,y\rangle\simeq F_2$ and $B=\langle x^m, y^n, x^{-1}y\rangle\simeq F_3$ and $\beta:B\to A$ is given by $\beta(x^m) = x^m$, $\beta(y^n) = y^n$ and $\beta(x^{-1}y) = yx^{-1}$.
\end{prop}

\begin{proof}
Presentation~\eqref{eq:2MN presentation} is the standard presentation of the HNN-extension $A*_{B,\beta}$ with the stable letter $b$.
\end{proof}

Alternatively, the splitting of $\text{Art}_{2MN}$ as above, can be deduced from a graph of spaces decomposition of $X_{2MN}$. 

\begin{prop}\label{prop:all even splitting}
Let $M = 2m$ and $N=2n$ be both even and $\geq 4$. Then $X_{2MN}$ is a graph of spaces $X(\Gamma)$ where $\Gamma$ is a single vertex with a single loop. The vertex space is the graph $X_A$, the edge space is the graph $X_B$ and the two maps $X_B\to X_A$ are given in the Figure~\ref{fig:both even maps}.
\end{prop}
\begin{proof}
Indeed, the map $p$ factors as $X_{2MN} \xrightarrow{\tilde p} S^1 = [-1,1]/{(-1\sim 1)} \to [0,1]$ where the second map is the absolute value, and where $\tilde p$ sends the loop $b$ isometrically onto $S^1$ and is extended linearly. 
By construction, the preimage $\tilde p^{-1}(t)$ is homeomorphic $X_B$ when $t\in (-1,1)$, and to $X_A$ when $t=1$.
In particular, $X_{2MN}$ can be expressed as a graph of spaces, induced by $\tilde p$, where the cellular structure of $S^1$ consists of a single vertex $v=1$ and a single edge $e$. Indeed, 
$$X_{2MN} = X_A\cup X_B\times[-1,1]/(x,-1)\sim f_{-}(x), (x,1)\sim f_{+}(x)$$
where $f_{-}, f_{+}:X_B\to X _A$ are the two maps obtained by ``pushing'' the graph $X_B$ in \Cref{fig:horizontal graphs}(a) ``upwards'' and ``downwards'' respectively. See Figure~\ref{fig:both even maps} for $f_{-}, f_{+}$ expressed as a composition of Stallings fold and a combinatorial immersion. 
\end{proof}

\begin{figure}
\begin{tikzpicture}
\begin{scope}[shift={(0,0)},decoration={
    markings,
    mark=at position 0.75 with {\arrow{>}}}]
\node[circle, draw, fill, inner sep = 0pt,minimum width = 2pt] (a) at (0,0) {};
\draw[RoyalBlue, thick] (a) to[out = 120, in =180] (0,1.2) to[out=0, in =60] (a);
\draw[RoyalBlue, thick] (a) to[out = 50, in =160] (0.8,1) to[out=-20, in =0] (a);
\draw[RoyalBlue, thick] (a) to[out = 130, in =20] (-0.8,1) to[out=200, in =180] (a);
\end{scope}
\begin{scope}[shift={(1.25,1)}]
\draw[->] (0,0) -- (0.8,0.5);
\end{scope}
\begin{scope}[shift={(3.5,1)},decoration={
    markings,
    mark=at position 0.65 with {\arrow{>}}}]
\node[circle, draw, fill, inner sep = 0pt,minimum width = 2pt] (a) at (0,0) {};
\node[circle, draw, fill, inner sep = 0pt,minimum width = 1pt] (b) at (0,1.2) {};
\draw[postaction={decorate},>=stealth, DarkOrchid, thick] (b) to[out = 180, in =120] (a);
\draw[postaction={decorate},>=stealth, LimeGreen, thick] (b) to[out=0, in=60] (a);
\draw[postaction={decorate},>=stealth,DarkOrchid, thick] (a) to[out = 50, in =160] (0.8,1) to[out=-20, in =0] (a);
\draw[postaction={decorate},>=stealth,LimeGreen, thick] (a) to[out = 130, in =20] (-0.8,1) to[out=200, in =180] (a);\node[LimeGreen, draw=none,fill=none] (x) at (-1, 0.3) {\tiny$m$};
\node[DarkOrchid, draw=none,fill=none] (x) at (1, 0.3) {\tiny$n$};
\end{scope}
\begin{scope}[shift={(1.25,0)}]
\draw[->] (0,0) -- (0.8,-0.5);
\end{scope}
\begin{scope}[shift={(3.5,-1)},decoration={
    markings,
    mark=at position 0.75 with {\arrow{>}}}]
\node[circle, draw, fill, inner sep = 0pt,minimum width = 2pt] (a) at (0,0) {};
\node[circle, draw, fill, inner sep = 0pt,minimum width = 1pt] (b) at (0,1.2) {};
\draw[postaction={decorate},>=stealth,LimeGreen, thick] (a) to[out = 120, in =180] (b);
\draw[postaction={decorate},>=stealth, DarkOrchid, thick] (a) to[out=60, in=0] (b);
\draw[postaction={decorate},>=stealth,DarkOrchid, thick] (a) to[out = 50, in =160] (0.8,1) to[out=-20, in =0] (a);
\draw[postaction={decorate},>=stealth,LimeGreen, thick] (a) to[out = 130, in =20] (-0.8,1) to[out=200, in =180] (a);\node[LimeGreen, draw=none,fill=none] (x) at (-1, 0.3) {\tiny$m$};
\node[DarkOrchid, draw=none,fill=none] (x) at (1, 0.3) {\tiny$n$};
\end{scope}
\begin{scope}[shift={(6,1.5)}]
\draw[->] (-1,0) -- (0,0);
\end{scope}
\begin{scope}[shift={(7,1)},decoration={
    markings,
    mark=at position 0.65 with {\arrow{>}}}]
\node[circle, draw, fill, inner sep = 0pt,minimum width = 2pt] (a) at (0,0) {};
\node[circle, draw, fill, inner sep = 0pt,minimum width = 1pt] (b) at (0,1.2) {};
\draw[postaction={decorate},>=stealth,LimeGreen, thick] (b) to[in = 120, out =240] (a);
\draw[postaction={decorate},>=stealth, DarkOrchid, thick] (b) to[in=60, out=-60] (a);
\draw[postaction={decorate},>=stealth,LimeGreen, thick] (a) to[in = 180, out =180] (b);
\draw[postaction={decorate},>=stealth,DarkOrchid, thick] (a) to[in = 0, out =0] (b);
\node[LimeGreen, draw=none,fill=none] (x) at (-0.9, 0.8) {\tiny$m-1$};
\node[DarkOrchid, draw=none,fill=none] (x) at (0.9, 0.8) {\tiny$n-1$};
\end{scope}
\begin{scope}[shift={(6,-0.5)}]
\draw[->] (-1,0) -- (0,0);
\end{scope}
\begin{scope}[shift={(7,-1)},decoration={
    markings,
    mark=at position 0.65 with {\arrow{>}}}]
\node[circle, draw, fill, inner sep = 0pt,minimum width = 2pt] (a) at (0,0) {};
\node[circle, draw, fill, inner sep = 0pt,minimum width = 1pt] (b) at (0,1.2) {};
\draw[postaction={decorate},>=stealth,LimeGreen, thick] (a) to[out = 120, in =240] (b);
\draw[postaction={decorate},>=stealth, DarkOrchid, thick] (a) to[out=60, in=-60] (b);
\draw[postaction={decorate},>=stealth,LimeGreen, thick] (b) to[out = 180, in =180] (a);
\draw[postaction={decorate},>=stealth,DarkOrchid, thick] (b) to[out = 0, in =0] (a);
\node[LimeGreen, draw=none,fill=none] (x) at (-0.8, 0.3) {\tiny$m-1$};
\node[DarkOrchid, draw=none,fill=none] (x) at (0.8, 0.3) {\tiny$n-1$};
\end{scope}
\begin{scope}[shift={(8,1.5)}]
\draw[->] (0,0) -- (0.8,-0.5);
\end{scope}
\begin{scope}[shift={(8,-0.5)}]
\draw[->] (0,0) -- (0.8,0.5);
\end{scope}
\begin{scope}[shift={(10,0)},decoration={
    markings,
    mark=at position 0.75 with {\arrow{>}}}]
\node[circle, draw, fill, inner sep = 0pt,minimum width = 2pt] (a) at (0,0) {};
\draw[postaction={decorate},DarkOrchid, thick,>=stealth] (a) to[out = 80, in =180] (0.5,1) to[out=0, in =10] (a);
\draw[postaction={decorate},LimeGreen, thick, >=stealth] (a) to[out = 100, in =0] (-0.5,1) to[out=180, in =170] (a) ;
\end{scope}
\end{tikzpicture}
\caption{The map $f_{-}:X_B\to X_B^- \to \bar X_B^{-}\to X_A$ (top), and $f_{+}:X_B\to X_B^+\to \bar X_B^{+}\to X_A$ (bottom).}\label{fig:both even maps}
\end{figure}
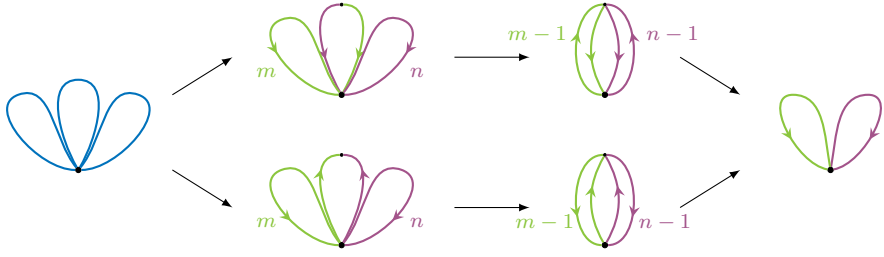

\begin{remark}\label{rem:C conjugate}
The subgroup $B$ and $\beta(B)$ are conjugate. See Figure~\ref{fig:both even maps}. Indeed, the graphs $\bar X_B^-$ and $\bar X_B^+$ are identical (but have different basepoints).
\end{remark}

\begin{exa}[Group $\text{Art}_{244}$]\label{exa:244}
In the case where $M=N=4$, Proposition~\ref{prop:all even splitting} provides the splitting of $\text{Art}_{244} = A*_{B, \beta}$ where $A= \langle x,y\rangle$ and $B=\langle x^2, y^2, x^{-1}y \rangle$, and $\beta:B\to B$ is given by $\beta(x^2) = x^2, \beta(y^2) = y^2, \beta(x^-1y) = yx^{-1}$. In particular, $B$ has index $2$ in $A$. This splitting was first proven by Squier \cite{Squier87}.
\end{exa}
\subsection{At least one odd}
We now assume that at least one $M,N$ is odd. We have the following description of the complex $X_{2MN}$.

\begin{prop}\label{prop:splitting of X}
The complex $X_{2MN}$ is a graph of spaces with the underlying graph is an interval the vertex spaces are graphs $X_A$ and $X _B$, and the edge space is $X_C$. 
The attaching map $X_C\to X_B$ is a double cover, 
and the attaching map $X_C\to X_A$ factors as $X_C\to \bar X_C\to X_A$, illustrated in \Cref{fig:CtoA}, 
where the first map is a homotopy equivalence and the second map is a combinatorial immersion.
\end{prop}

\begin{proof}Indeed, $X_{2MN}$ can be obtained as a union of $X_A, X_B$ and $X_C\times[0,1]$ where $X_C\times\{1\}$ is glued to $X_A$ and $X_C\times\{0\}$ is glued to $X_B$.
Note that the preimage $p^{-1}([0,\frac 12])$ is a union of ``half'' cylinders and Mobius strip, and its boundary is the graph $X_C$. The projection onto the center circle of the boundary of each cylinder or Mobius strip is a (connected or not) double cover of the center circle. It follows that $X_C\to X_B$ is a double cover. 
The map $X_C\to X_A$ is induced by ``pushing'' $X_C$ ``downwards'' and ``upwards'' onto $X_A$, and it can be described by the labelling of the right graphs in \Cref{fig:CtoA}. The factorization $X_C\to \bar X_C$ is obtained by performing Stallings folds. Note that the middle graphs in \Cref{fig:CtoA} are fully folded, provided that $m-1, n-1>0$, which is equivalent to the condition that $M,N\geq 4$. It follows that the map $\bar X_C\to X_A$ is a combinatorial immersion. Since the rank of $\pi_1X_C$ and $\pi_1\bar X_C$ are both equal $5$, the folding map is a homotopy equivalence.
\end{proof}

\begin{cor}\label{cor:at least one odd splitting}
Suppose at most one of $M,N$ is even and $M,N\geq 4$. Then $\text{Art}_{2MN}$ splits as a free product with amalgamation $A*_CB$ where $A=F_2$ and $B=F_3$, and $C=F_5$. 
\end{cor}

\begin{proof}
This directly follows from \Cref{prop:splitting of X}. We get that $\text{Art}_{2MN} = \pi_1X_{2MN} = A*_CB$ where 
$A = \pi_1X_A$, $B=\pi_1X_B$ and $C=\pi_1X_C$. From \Cref{fig:horizontal graphs}, we see that $\rank A = 2$, $\rank B = 3$ and $\rank C = 5$.
\end{proof}

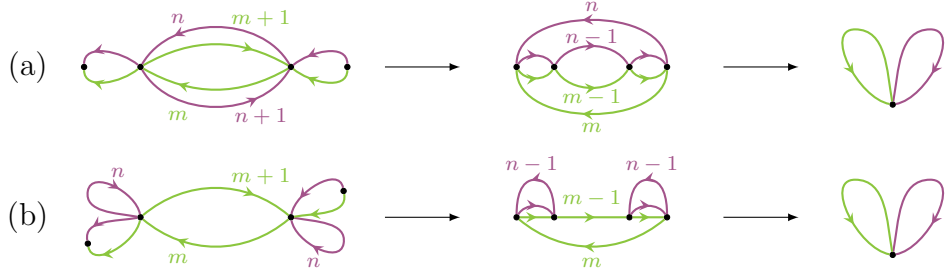
\begin{figure}
\begin{tikzpicture}
\begin{scope}[shift={(-5.25,0.5)},decoration={
    markings,
    mark=at position 0.75 with {\arrow{>}}}]
\node[draw=none,fill=none] (x) at (-1.5, 0) {(a)};
\node[circle, draw, fill, inner sep = 0pt,minimum width = 2pt] (-) at (0,0) {};
\node[circle, draw, fill, inner sep = 0pt,minimum width = 2pt] (+) at (2,0) {};
\node[circle, draw, fill, inner sep = 0pt,minimum width = 2pt] (a) at (-0.75,0) {};
\node[circle, draw, fill, inner sep = 0pt,minimum width = 2pt] (b) at (2.75,0) {};

\draw[LimeGreen, thick, postaction={decorate},>=stealth] (-) to[out = 30, in =150] (+);
\draw[LimeGreen, thick, postaction={decorate},>=stealth] (+) to[in = -30, out =210] (-);
\node[LimeGreen, draw=none,fill=none] (x) at (1.6, 0.65) {\tiny$m+1$};
\node[LimeGreen, draw=none,fill=none] (x) at (0.5, -0.65) {\tiny$m$};

\draw[DarkOrchid, thick, postaction={decorate},>=stealth] (+) to[in = 60, out =120] (-);
\draw[DarkOrchid, thick, postaction={decorate},>=stealth] (-) to[out = -60, in =240] (+);
\node[DarkOrchid, draw=none,fill=none] (x) at (0.5, 0.65) {\tiny$n$};
\node[DarkOrchid, draw=none,fill=none] (x) at (1.6, -0.65) {\tiny$n+1$};

\draw[DarkOrchid, thick, postaction={decorate},>=stealth] (-) to[out = 140, in =90] (a);
\draw[LimeGreen, thick, postaction={decorate},>=stealth] (-) to[in = 270, out =220] (a);
\draw[DarkOrchid, thick, postaction={decorate},>=stealth] (b) to[in = 40, out =90] (+);
\draw[LimeGreen, thick, postaction={decorate},>=stealth] (b) to[out = 270, in =-40] (+);
\end{scope}
\begin{scope}[shift={(-1,0.5)}]
\draw[->] (-1,0) -- (0,0);
\end{scope}
\begin{scope}[shift={(-0.25,0.5)}, decoration={
    markings,
    mark=at position 0.55 with {\arrow{>}}}]
\node[circle, draw, fill, inner sep = 0pt,minimum width = 2pt] (-) at (0,0) {};
\node[circle, draw, fill, inner sep = 0pt,minimum width = 2pt] (p) at (0.5,0) {};
\node[circle, draw, fill, inner sep = 0pt,minimum width = 2pt] (q) at (1.5,0) {};
\node[circle, draw, fill, inner sep = 0pt,minimum width = 2pt] (+) at (2,0) {};

\draw[DarkOrchid, thick, postaction={decorate},>=stealth] (+) to[in = 90, out =90] (-);
\node[DarkOrchid, draw=none,fill=none] (x) at (1, 0.8) {\tiny$n$};
\draw[LimeGreen, thick, postaction={decorate},>=stealth] (+) to[in = -90, out =-90] (-);
\node[LimeGreen, draw=none,fill=none] (x) at (1, -0.8) {\tiny$m$};

\node[DarkOrchid, draw=none,fill=none] (x) at (1, 0.4) {\tiny$n-1$};
\node[LimeGreen, draw=none,fill=none] (x) at (1, -0.4) {\tiny$m-1$};

\draw[DarkOrchid, thick, postaction={decorate},>=stealth] (-) to[in=120, out=60] (p);
\draw[DarkOrchid, thick, postaction={decorate},>=stealth] (q) to[in=120, out=60] (+);
\draw[DarkOrchid, thick, postaction={decorate},>=stealth] (p) to[in=120, out=60] (q);
\draw[LimeGreen, thick, postaction={decorate},>=stealth] (-) to[in=-120, out=-60] (p);
\draw[LimeGreen, thick, postaction={decorate},>=stealth] (q) to[in=-120, out=-60] (+);
\draw[LimeGreen, thick, postaction={decorate},>=stealth] (p) to[in=-120, out=-60] (q);
\end{scope}
\begin{scope}[shift={(3.5,0.5)}]
\draw[->] (-1,0) -- (0,0);
\end{scope}
\begin{scope}[shift={(4.75,0)},decoration={
    markings,
    mark=at position 0.75 with {\arrow{>}}}]
\node[circle, draw, fill, inner sep = 0pt,minimum width = 2pt] (a) at (0,0) {};
\draw[postaction={decorate},DarkOrchid, thick,>=stealth] (a) to[out = 80, in =180] (0.5,1) to[out=0, in =10] (a);
\draw[postaction={decorate},LimeGreen, thick,>=stealth] (a) to[out = 100, in =0] (-0.5,1) to[out=180, in =170] (a) ;
\end{scope}
\begin{scope}[shift={(-5.25,-1.5)},decoration={
    markings,
    mark=at position 0.75 with {\arrow{>}}}]
\node[draw=none,fill=none] (x) at (-1.5, 0) {(b)};
\node[circle, draw, fill, inner sep = 0pt,minimum width = 2pt] (-) at (0,0) {};
\node[circle, draw, fill, inner sep = 0pt,minimum width = 2pt] (+) at (2,0) {};
\node[circle, draw, fill, inner sep = 0pt,minimum width = 2pt] (a) at (-0.7, -0.35) {};
\node[circle, draw, fill, inner sep = 0pt,minimum width = 2pt] (b) at (2.7,0.35) {};

\draw[LimeGreen, thick, postaction={decorate},>=stealth] (-) to[out = 40, in =140] (+);
\node[LimeGreen, draw=none,fill=none] (x) at (0.5, -0.55) {\tiny$m$};
\draw[LimeGreen, thick, postaction={decorate},>=stealth] (+) to[in = -40, out =220] (-);
\node[LimeGreen, draw=none,fill=none] (x) at (1.6, 0.55) {\tiny$m+1$};

\draw[DarkOrchid, thick, postaction={decorate},>=stealth] (+) to[out = -10, in =80] (2.7,-0.35) to[out=260, in =280] (+);
\node[DarkOrchid, draw=none,fill=none] (y) at (-0.3, 0.6) {\tiny$n$};
\draw[DarkOrchid, thick, postaction={decorate},>=stealth] (-) to[out = 170, in =260] (-0.7,0.35) to[out=80, in =100] (-) ;
\node[DarkOrchid, draw=none,fill=none] (y) at (2.3, -0.6) {\tiny$n$};

\draw[LimeGreen, thick, postaction={decorate},>=stealth] (b) to[in = 10, out =280] (+);
\draw[DarkOrchid, thick, postaction={decorate},>=stealth] (b) to[out = 100, in =80] (+) ;
\draw[DarkOrchid, thick, postaction={decorate},>=stealth] (-) to[out = 190, in =100] (a);
\draw[LimeGreen, thick, postaction={decorate},>=stealth] (-) to[in=280, out =260] (a);
\end{scope}
\begin{scope}[shift={(-1,-1.5)}]
\draw[->] (-1,0) -- (0,0);
\end{scope}
\begin{scope}[shift={(-0.25,-1.5)},decoration={
    markings,
    mark=at position 0.55 with {\arrow{>}}}]
\node[circle, draw, fill, inner sep = 0pt,minimum width = 2pt] (-) at (0,0) {};
\node[circle, draw, fill, inner sep = 0pt,minimum width = 2pt] (p) at (0.5,0) {};
\node[circle, draw, fill, inner sep = 0pt,minimum width = 2pt] (q) at (1.5,0) {};
\node[circle, draw, fill, inner sep = 0pt,minimum width = 2pt] (+) at (2,0) {};

\draw[LimeGreen, thick, postaction={decorate},>=stealth] (-) to (p);
\draw[LimeGreen, thick, postaction={decorate},>=stealth] (p) to (q);
\draw[LimeGreen, thick, postaction={decorate},>=stealth] (q) to (+);
\node[LimeGreen, draw=none,fill=none] (x) at (1, 0.25) {\tiny$m-1$};

\draw[DarkOrchid, thick, postaction={decorate},>=stealth] (-) to[in=120, out=60] (p);
\draw[DarkOrchid, thick, postaction={decorate},>=stealth] (p) to[out=90, in=0] (0.25, 0.5) to[out=180, in=90] (-);
\node[DarkOrchid, draw=none,fill=none] (x) at (0.2, 0.7) {\tiny$n-1$};
\draw[DarkOrchid, thick, postaction={decorate},>=stealth] (q) to[in=120, out=60] (+);
\draw[DarkOrchid, thick, postaction={decorate},>=stealth] (+) to[out=90, in=0] (1.75, 0.5) to[out=180, in=90] (q);
\node[DarkOrchid, draw=none,fill=none] (x) at (1.8, 0.7) {\tiny$n-1$};

\node[LimeGreen, draw=none,fill=none] (x) at (1, -0.6) {\tiny$m$};
\draw[LimeGreen, thick, postaction={decorate},>=stealth] (+) to[in = -40, out =220] (-);

\end{scope}
\begin{scope}[shift={(3.5,-1.5)}]
\draw[->] (-1,0) -- (0,0);
\end{scope}
\begin{scope}[shift={(4.75,-2)},decoration={
    markings,
    mark=at position 0.75 with {\arrow{>}}}]
\node[circle, draw, fill, inner sep = 0pt,minimum width = 2pt] (a) at (0,0) {};
\draw[postaction={decorate},DarkOrchid, thick,>=stealth] (a) to[out = 80, in =180] (0.5,1) to[out=0, in =10] (a);
\draw[postaction={decorate},LimeGreen, thick, >=stealth] (a) to[out = 100, in =0] (-0.5,1) to[out=180, in =170] (a) ;
\end{scope}
\end{tikzpicture}
\caption{The maps $X_C\to\bar X_C \to X_A$ when (a): $M=2m+1,N=2n+1$ are both odd; and (b): $M=2m+1,N=2n$. }\label{fig:CtoA}
\end{figure}

\subsection{Splittings of spherical Artin groups}
All spherical Artin groups have non-trivial center and their cohomological dimension is equal to the number of standard generators \cite{Deligne72}, \cite{BrieskornSaito72}. We now give a characterization of graph of finite rank free groups with non-trivial center. This allows us to deduce that the only spherical Artin groups that split as graphs of finite rank free groups are the dihedral Artin groups (i.e.\ on two generators).

\begin{prop}\label{prop:cyclic center} Let $G = G(\Gamma)$ be a finite graph of free groups. 
Then the following conditions are equivalent.
\begin{enumerate}
\item The center of $G$ is non-trivial,
\item $G(\Gamma)$ satisfies one of the following:
\begin{enumerate}
\item all the vertex groups and edge groups are $\mathbb Z$ and for every $h\in G_v$ and $g\in \pi_1\Gamma$, if $g^{-1}hg\in G_v$, then $g^{-1}hg = h$, 
\item $\pi_1\Gamma\simeq \mathbb Z = \langle g \rangle$, all vertex and edge groups are isomorphic, and  there exists $n\in\mathbb Z-\{0\}$ such that for every $v\in V(\Gamma)$ and $h\in G_v$ we have $g^{-n}hg^n = h$.
\end{enumerate} 
\end{enumerate}
In particular, every finite graph of free groups with non-trivial center is virtually $F\times \mathbb Z$ where $F$ is a free group.
\end{prop}

\begin{proof} 
We first show that $G(\Gamma)$ satisfying (2) has non-trivial center. 
First suppose that Condition (2a) holds.
Then all the inclusions $G_e\to G_v$ are inclusions of finite index subgroup. 
Since there are finitely many vertex and edge groups, 
their intersection $H$ is non-empty and has finite index in each vertex and edge group. 
The condition that for every $h\in G_v$ and $g\in \pi_1\Gamma$ such that $g^{-1}hg\in G_v$ we have $g^{-1}hg = h$, 
implies that $H$ is central in $G$, and $H\lhd G$. The quotient $G/H$ is a finite graph of finite groups, so it is virtually free. It follows that $G$ is virtually $F\times \mathbb Z$ for some free group $F$.
Now suppose that Condition (2b) holds. 
Then $G\simeq G_v\rtimes \pi_1\Gamma$ where $\pi_1\Gamma= \mathbb Z = \langle g \rangle$ and $g^n$ is central in $G$. The subgroup $G_v\times \langle g^n\rangle = F\times \mathbb Z$ has finite index in $G$.

We now prove the other direction. Consider a graph of groups $G(\Gamma)$ with the fundamental group $G$ whose center is non-trivial. 
We have $Z(G) \cap G_v \subset Z(G_v)$ for every vertex $v$, 
so either $G_v$ is infinite cyclic, or $G_v$ intersects the center trivially. 
The same is true for edge groups. 
If any $G_v$ intersects the center non-trivially, 
then all the vertex group and edge groups must intersect the center nontrivially, 
and in particular they are all infinite cyclic.  Then for $h\in Z(G)\cap G_v$, and $g\in\pi_1\Gamma$, we must have $g^{-1}hg = h$, so $G(\Gamma)$ satisfies Condition (2a). 

Now suppose that every vertex and edge group intersects the center trivially. 
The center must be contained in the center of $\pi_1 \Gamma$. 
Since $\pi_1\Gamma$ is a free group, it must be equal $\mathbb Z = \langle g\rangle$, 
and there must exist a power $g^n$ such that $g^{-n}hg^n = h$ for every $h\in G_v$ and every $v\in V(\Gamma)$. 
It follows that all vertex and edge groups are isomorphic and Condition (2b) is satisfied.

\end{proof}

\begin{cor}
The only irreducible spherical Artin group that split as non-trivial graphs of free groups,
 are the dihedral Artin groups.
\end{cor}
\begin{proof} Let $\text{Art}_M$ be a dihedral Artin group with the presentation $\langle a,b\mid (a,b)_M = (b,a)_M\rangle$. 
Let $M=2m$ be even, and let $x = ab$. 
Then $\text{Art}_M\simeq \langle a,x \mid ax^ma^{-1} = x^m\rangle$. 
In particular, $\text{Art}_M = \langle x\rangle*_{\langle x^m\rangle} = \mathbb Z*_{\mathbb Z}$.
Now if $M=2m+1$, let $x = ab$ and $y= (a,b)_M$. Then $\text{Art}_M \simeq \langle x,y \mid x^M = y^2\rangle$. 
In particular, $\text{Art}_M\simeq \langle x\rangle *_{\langle x^M\rangle}\langle y\rangle = \mathbb Z*_{\mathbb Z}\mathbb Z$. 
By \Cref{prop:cyclic center}, spherical Artin groups on more than two generators do not split as amalgamated products of HNN extension of nontrivial finite rank free groups, since their (virtual) cohomological dimension is greater than $2$, which is the (virtual) cohomological dimension of $F\times \mathbb Z$.
\end{proof}

\subsection{Splittings of three generator Artin groups with $\infty$ labels}
To complete the picture, we prove that the remaining three generator Artin groups, 
i.e.\ those with at least one $\infty$ label, also split as graphs of finite rank free groups. 
The Artin group $\text{Art}_{\infty\infty \infty}$ is the free group on three generators. 
The group $\text{Art}_{M\infty\infty} = \text{Art}_{M}*\mathbb Z$ can be described as 
$  \langle x, c\rangle*_{x^M=y^2}\langle y\rangle = F_2*_{\mathbb Z}\mathbb Z$ where $x = ab$ and $y = (a,b)_M$. 
Finally, for the Artin group $\text{Art}_{MN\infty}$ we can use the presentation (\ref{eq:2MN presentation}) skipping the relation $bx^{-1}yb^{-1}=yx^{-1}$, i.e.
\[ \text{Art}_{MN\infty} = \langle x,y,b \mid r_M(b, x), r_N(b,y)\rangle\]
We get that $\text{Art}_{MN\infty}$ splits as 
\begin{itemize}
\item $A*_{B, \beta}$ where $A = \langle x,y \rangle \simeq F_2$, $B= \langle x^m, y^n\rangle\simeq F_2$ $\beta = \text{id}_C$, when $M=2m$ and $N=2n$,
\item $A*_CB$ where $A = \langle x, y\rangle\simeq F_2$, $B\simeq F_2$ and $C\simeq F_3$, when at least one of $M,N$ is odd. The splitting is obtained in the same way as in the case of $\text{Art}_{2MN}$.
\end{itemize}

\section{Residual finiteness}\label{sec:rf}
In the section we prove \Cref{thm:rf}. We do so separately in the case where $M,N$ are both even (\Cref{thm:rf all even}), and where exactly one of $M,N$ is odd (\Cref{thm:rf one odd}). We start with recalling a criterion for residual finiteness of amalgamated products and HNN extensions of finite rank free groups, proven in \cite{JankiewiczArtinRf}, which relies on Wise's result on residual finiteness of \emph{algebraically clean} graphs of free groups \cite{WisePolygons}. In the second subsection we compute the intersections of conjugates of the amalgamating subgroup in the factor groups. Finally, we give proofs of the main theorems.

We start with a motivating example.
\begin{exa}[Group $\text{Art}_{244}$]\label{exa:244rf}
By \Cref{exa:244} the group $\text{Art}_{244}$ fits in  the following short exact sequence of groups 
$$1\to C\to \text{Art}_{244} \to \mathbb Z/2\mathbb Z*\mathbb Z\to 1.$$
In particular, $\text{Art}_{244}$ is a (finite rank free group)-by-(virtually free group), and so it is virtually (finite rank free group)-by-(free group). The residual finiteness of $\text{Art}_{244}$ follows from the fact that every split extension of a finitely generated residually finite group by residually finite group is residually finite \cite{Malcev83}.
\end{exa}

\subsection{Criteria for residual finiteness}
Recall that a subgroup $C$ is \emph{malnormal} in a group $A$, if $C\cap g^{-1}Cg = \{1\}$ for every $g\in A-C$. 
Similarly, $C$ is \emph{almost malnormal} in $A$, if $|C\cap g^{-1}Cg|<\infty$ for every $g\in A-C$. 
More generally, a collection of subgroup $\{C_i\}_{i\in I}$ is an \emph{almost malnormal} collection, 
if $|C_i\cap gC_jg^{-1}|<\infty$ whenever $g\notin C_i$ or $i\neq j$.

We now assume that the inclusion of free groups $C\to A$ is induced by a map $f:X_C\to X_A$ of graphs, and the automorphism $\beta:C\to C$ is induced by some graph automorphism $X_C\to X_C$.
The following theorem was proven in \cite{JankiewiczArtinRf}. 

\begin{thm}[{\cite[Thm 2.9]{JankiewiczArtinRf}}]\label{thm:conditions for rf am prod}
Let $\hat f:\hat X_C\to \hat X_A$ be a map of $2$-complexes that restricted to the $1$-skeletons is equal $f$, and let $\pi:A \to \hat A$ be the natural quotient induced by the inclusion $X_A\hookrightarrow \hat X_A$ of the $1$-skeleton. Suppose that the following conditions hold.
\begin{enumerate}
\item $\hat A$ is a locally quasiconvex, virtually special hyperbolic group.
\item $\pi(C) = \pi_1\hat X_C$ and that the lift of $\hat f$ to the universal covers is an embedding.
\item $\pi(C)$ is almost malnormal in $\hat A$.
\item $\beta$ projects to an isomorphism $\pi(C)\to \pi(C)$.
\end{enumerate}
Then $D(A,C,\beta)$ is residually finite.
\end{thm}
The theorem above is a combination of Thm 2.9 and Lem 2.6 in \cite{JankiewiczArtinRf}. 
Condition (2) in the statement of \cite[Thm 2.9]{JankiewiczArtinRf} is that $\pi(C)$ is \emph{malnormal} in $\bar A$. However, the proof is identical when we replace it with \emph{almost malnormal}. Indeed, the Bestvina-Feighn combination theorem \cite{BestvinaFeighn92}, as well as the Hsu-Wise combination theorem \cite{HsuWiseCubulatingMalnormal} only require almost malnormality.

We now state a version for HNN extension. 
Similarly as above, combining Thm 2.12 and Lem 2.6 from \cite{JankiewiczArtinRf}, we obtain the following.

\begin{thm}[{\cite[Thm 2.12]{JankiewiczArtinRf}}]\label{thm:conditions for rf hnn}
Let $\hat f_{-}, \hat f_{+}:\hat X_B\to \hat X_A$ be two maps of $2$-complexes that restricted to the $1$-skeletons are equal to $f_{-}, f_{+}$ respectively, and let $\pi:A \to \hat A$ be the natural quotient induced by the inclusion $X_A\hookrightarrow \hat X_A$. Suppose that the following conditions hold.
\begin{enumerate}
\item $\hat A$ is a locally quasiconvex, virtually special hyperbolic group.
\item $\pi(B) = \hat f_{-*}( \pi_1\hat X_B)$, and $\pi\left(\beta(B)\right) = \hat f_{+*}( \pi_1\hat X_B)$ and the lifts of $\hat f_{-}, \hat f_{+}$ to the universal covers are both embeddings.
\item The collection $\{\pi(B), \pi\left(\beta(B)\right)\}$ is almost malnormal in $\hat A$.
\item $\beta:B\to \beta(B)$ projects to an isomorphism $\pi(B)\to \pi\left(\beta(B)\right)$.
\end{enumerate}
Then $A*_{B,\beta}$ is residually finite.
\end{thm}
\subsection{Intersection of the conjugates of the amalgamating subgroup}

Let $G = \text{Art}_{2MN}$ where $M,N\geq 4$ and at least one of them is odd. Let $A, B, C$ be free groups of ranks $2,3,5$ respectively, provided by \Cref{cor:at least one odd splitting}. In this section we describe intersections $C\cap g^{-1}Cg$ where $g\in A$.

\begin{prop}\label{prop:fiber product one odd}
 Let $M= 2m+1$ be odd, $N=2n$ even, and let $A,B,C$ be as in \Cref{cor:at least one odd splitting}. Let $g\in A-C$. Then the intersection $C\cap g^{-1}Cg$ is one of: $\langle x^{2m+1}, y^n, x^{-1}y \rangle$, $\langle x^{2m+1}, y^n, yx^{-1} \rangle$, $\langle x^{2m+1}, y^n\rangle$,  $\langle x^{2m+1}\rangle$,  $\langle y^n\rangle$. 
\end{prop}

\begin{figure}
\begin{tikzpicture}[decoration={
   markings,
    mark=at position 0.55 with {\arrow{>}}}]
\begin{scope}
\node[draw=none,fill=none] (x) at (-1, 0) {(a)};
\node[circle, draw, fill, inner sep = 0pt,minimum width = 2pt] (-) at (0,0) {};
\node[circle, draw, fill, inner sep = 0pt,minimum width = 2pt] (+) at (2,0) {};

\draw[DarkOrchid, thick, postaction={decorate},>=stealth] (+) to[in = 90, out =90] (-);
\draw[DarkOrchid, thick, postaction={decorate},>=stealth] (-) to[in = 150, out =30] (+);
\node[DarkOrchid, draw=none,fill=none] (x) at (0.9, 0.8) {\tiny$n-1$};
\draw[LimeGreen, thick, postaction={decorate},>=stealth] (+) to[in = -90, out =-90] (-);
\draw[LimeGreen, thick, postaction={decorate},>=stealth] (-) to[in = -150, out =-30] (+);
\node[LimeGreen, draw=none,fill=none] (x) at (0.9, -0.8) {\tiny$2m+1$};
\end{scope}
\begin{scope}[shift={(5,0)}]
\node[draw=none,fill=none] (x) at (-1, 0) {(b)};
\node[circle, draw, fill, inner sep = 0pt,minimum width = 2pt] (-) at (0,0) {};
\node[circle, draw, fill, inner sep = 0pt,minimum width = 2pt] (p) at (0.5,0) {};
\node[circle, draw, fill, inner sep = 0pt,minimum width = 2pt] (+) at (2,0) {};

\draw[DarkOrchid, thick, postaction={decorate},>=stealth] (+) to[in = 90, out =90] (-);
\node[DarkOrchid, draw=none,fill=none] (x) at (0.9, 0.8) {\tiny$n$};
\draw[LimeGreen, thick, postaction={decorate},>=stealth] (+) to[in = -90, out =-90] (-);
\node[LimeGreen, draw=none,fill=none] (x) at (0.9, -0.8) {\tiny$m$};

\node[DarkOrchid, draw=none,fill=none] (x) at (1.3, 0.2) {\tiny$n$};
\node[LimeGreen, draw=none,fill=none] (x) at (1.3, -0.2) {\tiny$m$};

\draw[DarkOrchid, thick, postaction={decorate},>=stealth] (-) to[in=120, out=60] (p);
\draw[DarkOrchid, thick, postaction={decorate},>=stealth] (p) to[in=120, out=60] (+);
\draw[LimeGreen, thick, postaction={decorate},>=stealth] (-) to[in=-120, out=-60] (p);
\draw[LimeGreen, thick, postaction={decorate},>=stealth] (p) to[in=-120, out=-60] (+);
\end{scope}
\end{tikzpicture}
\caption{A connected component of $\bar X_C\otimes_{X_A} \bar X_C$ when (a): $M=2m+1,N=2n$ , and (b): $M=2m+1,N=2n+1$ (right).}\label{fig:fiber product}
\end{figure}
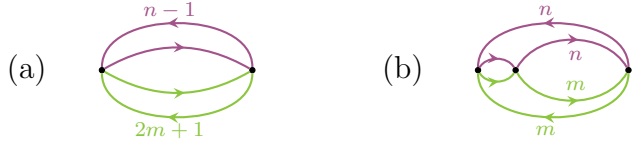

\begin{proof}
Let $X_A, \bar X_C$ be as in \Cref{fig:CtoA}(b). 
By \Cref{thm:fiber product} the conjugates $C\cap g^{-1}Cg$ can be represented by the connected components of the fiber product $\bar X_C\otimes_{X_A} \bar X_C$.
Let $x_1, x_2, x_3, x_4$ be the four vertices of valence $4$ in $\bar X_C$ such that $x_1, x_2$ belong to the same $y$-cycle and $x_3, x_4$ belong to the same $y$-cycle, and so that the ordering if the indices is consistent with the order of the vertices on the $x$-cycle. 
Then the connected component of $\bar X_C\otimes_{X_A} \bar X_C$ containining one of $(x_i, x_j)$ is
\begin{itemize}
\item the graph in \Cref{fig:fiber product}(a), if $|i-j| = 2$, in which case $C\cap g^{-1}Cg$ is $\langle x^{2m+1}, y^n, x^{-1}y \rangle$ or $\langle x^{2m+1}, y^n, yx^{-1} \rangle$, or
\item a bouquet of two circles, labelled by $x^{2m+1}$ and $y^n$ otherwise, in which case $C\cap g^{-1}Cg$ is $\langle x^{2m+1}, y^n\rangle$.
\end{itemize}
Every other non-diagonal connected component of $\bar X_C\otimes_{X_A} \bar X_C$ is either trivial or a single circle, which is labelled by either $x^{2m+1}$ or $y^n$, in which case $C\cap g^{-1}Cg$ is $\langle x^{2m+1}\rangle$ or  $\langle y^n\rangle$ respectively.
\begin{com}
conjugacy class versus actual group
\end{com}
\end{proof}

We finish with the following observation regarding the case where $M,N$ are both odd.
\begin{remark}\label{rem:both odd} Let $M= 2m+1, N=2n+1$ be both odd, and let $A,B,C$ be as in \Cref{cor:at least one odd splitting}. 
Let $g\in A-C$. Then the intersection $C\cap g^{-1}Cg$ is one of: $\langle x^{2m+1}, y^n, x^{-1}y \rangle$, $\langle x^{2m+1}, y^n\rangle$,  $\langle x^{2m+1}\rangle$,  $\langle y^n\rangle$. Let $X_A, \bar X_C$ be as in \Cref{fig:CtoA}. 
Let $x_1, x_2, x_3, x_4$ be the four vertices of valence $4$ in $\bar X_C$ ordered consistently with the orientation of the $x$-cycle such that $x_1$ and $x_4$ are at distance $m$ in the $x$-cycle.
Then the connected component of $\bar X_C\otimes_{X_A} \bar X_C$ containining vertices $(x_1, x_3), (x_2, x_4), (x_4, x_1)$ (or vertices $(x_3, x_1), (x_4, x_2), (x_1, x_4)$) looks like the right graph in \Cref{fig:fiber product}.
\end{remark}

\subsection{Proof of residual finiteness}
We now use  \Cref{thm:conditions for rf hnn} to prove that $\text{Art}_{2MN}$ where $M,N$ are even and equal at least $4$ is residually finite.
\begin{thm}\label{thm:rf all even}
Let $M = 2m$ and $N=2n$ be both even and $\geq 4$. Then $\text{Art}_{2MN}$ is residually finite.
\end{thm}
\begin{proof}
The case where $M=N=4$ is proven in \Cref{exa:244rf}, so we assume that at least one of $M,N$, say $M$, is at least $6$.
By \Cref{prop:all even splitting}, $\text{Art}_{2MN}$ splits as an HNN-extension $A*_{B,\beta}$ where $A=\langle x,y\rangle$ and $B=\langle x^m, y^n, x^{-1}y\rangle $, 
and $\beta:B\to A$ is given by $x^m\mapsto x^m$, $y^n\mapsto y^n$ and $x^{-1}y \mapsto yx^{-1}$.
We deduce residual finiteness of $\text{Art}_{2MN}$ from \Cref{thm:conditions for rf hnn}. We now check that all its assumptions are satsified.

Let $\hat A =\langle x,y \mid, x^m, y^n, (x^{-1}y)^p\rangle$ where $p\geq 7$, and let $\hat X_A$ be the presentation complex of $\hat A$. Let $\pi:A\to \hat A$ be the natural quotient. 
Since $m\geq 3$, the group $\hat A$ is a hyperbolic von Dyck triangle group, and in particular Condition (1) of \Cref{thm:conditions for rf hnn} is satisfied. 

The image $\pi(B)$ is a finite cyclic group $\mathbb Z/p$ of order $p$ generated by $x^{-1}y$, 
and the image $\pi\left(\beta(B)\right)$ is a copy of $\mathbb Z/p$ generated by $yx^{-1}$.
Since $\pi(B), \pi\left(\beta(B)\right)$ are finite groups, they form an almost malnormal collection in $\hat A$, 
so Condition (3) in \Cref{thm:conditions for rf hnn} is satisfied. 
Let $\hat X_B$ be obtained from $X_B$ by attaching a $2$-cell to each of the left and the right loop of $X_B$ (left and right in Figure~\ref{fig:both even maps}) via a $1$-to-$1$ map (corresponding to the relations $x^m, y^n$), and one $2$-cell to the middle loop via a $p$-to-$1$ map (corresponding to the relation $(x^{-1}y)^p$). 
It is immediate that Conditions (2) and (4) of \Cref{thm:conditions for rf hnn} holds. 
See Figure~\ref{fig:attaching 2-cells}. The proof is complete.

\begin{figure}
\begin{tikzpicture}
\begin{scope}[shift={(0,0)},decoration={
    markings,
    mark=at position 0.75 with {\arrow{>}}}]
\path[fill=RoyalBlue!30] (0,0) to[out = 120, in =180] (0,1.2) to[in=60, out=0] cycle;
\node[circle, draw, fill=white, inner sep = 0.4pt,minimum width = 2pt] (a) at (0,0.9) {};
\node[circle, draw, fill, inner sep = 0pt,minimum width = 2pt] (a) at (0,0) {};
\draw[RoyalBlue, thick] (a) to[out = 120, in =180] (0,1.2) to[out=0, in =60] (a);
\draw[RoyalBlue, thick, fill=RoyalBlue!30] (a) to[out = 50, in =160] (0.8,1) to[out=-20, in =0] (a);
\draw[RoyalBlue, thick,  fill=RoyalBlue!30] (a) to[out = 130, in =20] (-0.8,1) to[out=200, in =180] (a);
\end{scope}
\begin{scope}[shift={(1.25,1)}]
\draw[->] (0,0) -- (0.8,0.5);
\end{scope}
\begin{scope}[shift={(3.5,1)},decoration={
    markings,
    mark=at position 0.65 with {\arrow{>}}}]
\path[fill=gray!30] (0,0) to[out = 120, in =180] (0,1.2) to[in=60, out=0] cycle;
\node[circle, draw, fill=white, inner sep = 0.4pt,minimum width = 2pt] (a) at (0,0.9) {};
\node[circle, draw, fill, inner sep = 0pt,minimum width = 2pt] (a) at (0,0) {};
\node[circle, draw, fill, inner sep = 0pt,minimum width = 1pt] (b) at (0,1.2) {};
\draw[postaction={decorate},>=stealth, DarkOrchid, thick] (b) to[out = 180, in =120] (a);
\draw[postaction={decorate},>=stealth, LimeGreen, thick] (b) to[out=0, in=60] (a);
\draw[postaction={decorate},>=stealth,DarkOrchid, thick,  fill=DarkOrchid!30] (a) to[out = 50, in =160] (0.8,1) to[out=-20, in =0] (a);
\draw[postaction={decorate},>=stealth,LimeGreen, thick,fill=LimeGreen!30] (a) to[out = 130, in =20] (-0.8,1) to[out=200, in =180] (a);\node[LimeGreen, draw=none,fill=none] (x) at (-1, 0.3) {\tiny$m$};
\node[DarkOrchid, draw=none,fill=none] (x) at (1, 0.3) {\tiny$n$};
\end{scope}
\begin{scope}[shift={(1.25,0)}]
\draw[->] (0,0) -- (0.8,-0.5);
\end{scope}
\begin{scope}[shift={(3.5,-1)},decoration={
    markings,
    mark=at position 0.75 with {\arrow{>}}}]
\path[fill=gray!30] (0,0) to[out = 120, in =180] (0,1.2) to[in=60, out=0] cycle;
\node[circle, draw, fill=white, inner sep = 0.4pt,minimum width = 2pt] (a) at (0,0.9) {};
\node[circle, draw, fill, inner sep = 0pt,minimum width = 2pt] (a) at (0,0) {};
\node[circle, draw, fill, inner sep = 0pt,minimum width = 1pt] (b) at (0,1.2) {};
\draw[postaction={decorate},>=stealth,LimeGreen, thick, name path =z] (a) to[out = 120, in =180] (b);
\draw[postaction={decorate},>=stealth, DarkOrchid, thick, name path =zz] (a) to[out=60, in=0] (b);\draw[postaction={decorate},>=stealth,DarkOrchid, thick, fill=DarkOrchid!30] (a) to[out = 50, in =160] (0.8,1) to[out=-20, in =0] (a);
\draw[postaction={decorate},>=stealth,LimeGreen, thick,fill=LimeGreen!30] (a) to[out = 130, in =20] (-0.8,1) to[out=200, in =180] (a);\node[LimeGreen, draw=none,fill=none] (x) at (-1, 0.3) {\tiny$m$};
\node[DarkOrchid, draw=none,fill=none] (x) at (1, 0.3) {\tiny$n$};
\end{scope}
\begin{scope}[shift={(6,1.5)}]
\draw[->] (-1,0) -- (0,0);
\end{scope}
\begin{scope}[shift={(7,1)},decoration={
    markings,
    mark=at position 0.65 with {\arrow{>}}}]
\path[fill=gray!30] (0,1.2) to[out = 240, in =120] (0,0) to[in=-60, out=60] cycle;
\node[circle, draw, fill=white, inner sep = 0.4pt,minimum width = 2pt] (a) at (0,0.6) {};
\node[circle, draw, fill, inner sep = 0pt,minimum width = 2pt] (a) at (0,0) {};
\node[circle, draw, fill, inner sep = 0pt,minimum width = 1pt] (b) at (0,1.2) {};
\draw[postaction={decorate},>=stealth,LimeGreen, thick,name path=y] (b) to[in = 120, out =240] (a);
\draw[postaction={decorate},>=stealth, DarkOrchid, thick, name path=x] (b) to[in=60, out=-60] (a);
\draw[postaction={decorate},>=stealth,LimeGreen, thick, name path=yy] (a) to[in = 180, out =180] (b);
\draw[postaction={decorate},>=stealth,DarkOrchid, thick, name path=xx] (a) to[in = 0, out =0] (b);
\node[LimeGreen, draw=none,fill=none] (x) at (-0.9, 0.8) {\tiny$m-1$};
\node[DarkOrchid, draw=none,fill=none] (x) at (0.9, 0.8) {\tiny$n-1$};
\tikzfillbetween[of=y and yy]{LimeGreen, opacity=0.3};
\tikzfillbetween[of=x and xx]{DarkOrchid, opacity=0.3};

\end{scope}
\begin{scope}[shift={(6,-0.5)}]
\draw[->] (-1,0) -- (0,0);
\end{scope}
\begin{scope}[shift={(7,-1)},decoration={
    markings,
    mark=at position 0.65 with {\arrow{>}}}]
\path[fill=gray!30] (0,1.2) to[out = 240, in =120] (0,0) to[in=-60, out=60] cycle;
\node[circle, draw, fill=white, inner sep = 0.4pt,minimum width = 2pt] (a) at (0,0.6) {};
\node[circle, draw, fill, inner sep = 0pt,minimum width = 2pt] (a) at (0,0) {};
\node[circle, draw, fill, inner sep = 0pt,minimum width = 1pt] (b) at (0,1.2) {};
\draw[postaction={decorate},>=stealth,LimeGreen, thick, name path=y] (a) to[out = 120, in =240] (b);
\draw[postaction={decorate},>=stealth, DarkOrchid, thick, name path=x] (a) to[out=60, in=-60] (b);
\draw[postaction={decorate},>=stealth,LimeGreen, thick,name path=yy] (b) to[out = 180, in =180] (a);
\draw[postaction={decorate},>=stealth,DarkOrchid, thick, name path=xx] (b) to[out = 0, in =0] (a);
\tikzfillbetween[of=y and yy]{LimeGreen, opacity=0.3};
\tikzfillbetween[of=x and xx]{DarkOrchid, opacity=0.3};

\node[LimeGreen, draw=none,fill=none, name path=z] (x) at (-0.8, 0.3) {\tiny$m-1$};
\node[DarkOrchid, draw=none,fill=none, name path=zz] (x) at (0.8, 0.3) {\tiny$n-1$};
\end{scope}
\begin{scope}[shift={(8,1.5)}]
\draw[->] (0,0) -- (0.8,-0.5);
\end{scope}
\begin{scope}[shift={(8,-0.5)}]
\draw[->] (0,0) -- (0.8,0.5);
\end{scope}
\begin{scope}[shift={(10,0)},decoration={
    markings,
    mark=at position 0.75 with {\arrow{>}}}]
\node[circle, draw, fill, inner sep = 0pt,minimum width = 2pt] (a) at (0,0) {};
\draw[postaction={decorate},DarkOrchid, thick,>=stealth, fill=DarkOrchid!30] (a) to[out = 80, in =180] (0.5,1) to[out=0, in =10] (a);
\draw[postaction={decorate},LimeGreen, thick, >=stealth, fill=LimeGreen!30] (a) to[out = 100, in =0] (-0.5,1) to[out=180, in =170] (a) ;
\draw[LimeGreen,shift={(1.5,1.5)},fill=gray!30, thick] (0:0.5) \foreach \x in {25.714,51.428,...,359} {
            -- (\x:0.5)
        } -- cycle (90:0.5);
\foreach \x in {0,51.428,...,359} {
\draw[DarkOrchid, shift={(1.5,1.5)}, thick] (\x:0.5) -- (\x+25.714:0.5);
        }
\draw[->] (0.8, 1.5) to[out=180, in=90] (0,1);
\end{scope}
\end{tikzpicture}
\caption{The map $\hat f_{-}:\hat X_B\to\hat{ \bar X}_B^-\to \hat X_A$ (top), and $\hat f_{+}:\hat X_B\to\hat {\bar X}_B^+\to\hat X_A$ (bottom). White nodes are contained in the $2$-cells whose boundary is mapped $p$-to-$1$.}\label{fig:attaching 2-cells}
\end{figure}
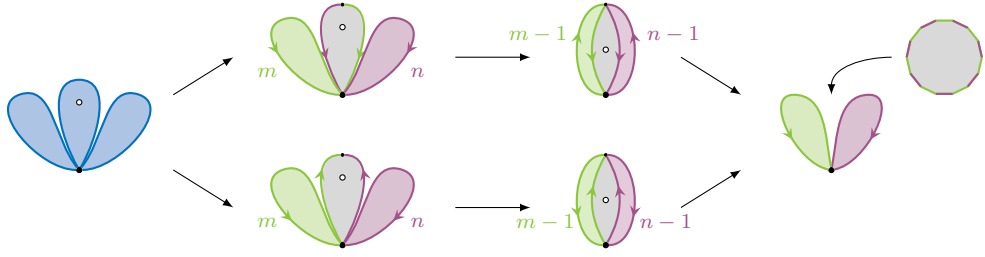
\end{proof}

Similarly, we use \Cref{thm:conditions for rf am prod} to prove that $\text{Art}_{2MN}$ where one of $M,N$ is odd, is residually finite.
\begin{thm}\label{thm:rf one odd}
Let $M = 2m+1$ and $N=2n$ be both $\geq 4$. Then $\text{Art}_{2MN}$ is residually finite.
\end{thm}
\begin{proof}
We use \Cref{thm:conditions for rf am prod}. 
Let $\hat A = \langle x,y\mid x^{2m+1}, y^n, (x^{-1}y)^p \rangle$ where $p\geq 6$ and let $\hat X_A$ be its presentation $2$-complex.

Since $n\geq 2$ and $2m+1\geq 5$, the group $\hat A$ is a hyperbolic von Dyck triangle group, and in particular, $\hat A$ satisfies Condition (1) of \Cref{thm:conditions for rf am prod}. Let $\pi:A\to \hat A$ be the natural quotient. 
Let $\hat {\bar X}_C$ be a $2$-complex obtained from $\bar X_C$ by attaching five $2$-cells: one along the unique cycle labelled $x^{2m+1}$, one along each of the two cycles labelled by $y^n$, and one along each of two cycles $xy^{-1}$ via a $p$-to-one map. See \Cref{fig:attaching 2-cells one odd}. 
In \Cref{lem:von dyck lemma} (below) we verify that Condition (2) of \Cref{thm:conditions for rf am prod} is satisfied. 

By \Cref{prop:fiber product one odd}, the intersection of distinct conjugates of $\pi(C)$ in $\hat A$ is either $\mathbb Z/p$ or trivial. 
In particular, $\pi(C)$ is almost malnormal in $\hat A$, so Condition (3) of \Cref{thm:conditions for rf am prod} is satisfied. 
Finally, we note that the $2$-cells of $\hat {\bar X}_C$ can be pulled back via $X_C\to \bar X_C$, 
and are preserved under the (unique) nontrivial deck transformation of $X_C\to X_B$. 
See \Cref{fig:attaching 2-cells one odd}. Condition (4) of \Cref{thm:conditions for rf am prod} follows. 
This completes the proof.

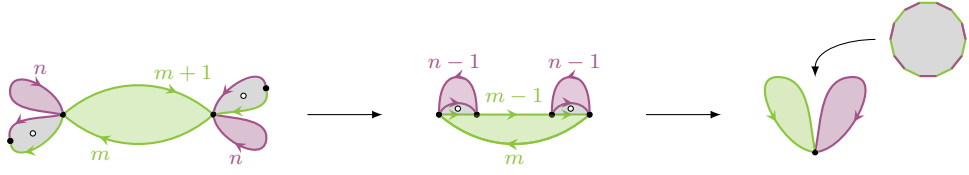
\begin{figure}
\begin{tikzpicture}
\begin{scope}[shift={(-5.25,-1.5)},decoration={
    markings,
    mark=at position 0.75 with {\arrow{>}}}]
\path[fill=gray!30] (2.7,0.35) to[in = 10, out =290] (2,0) to[in=90, out=100] (2.7,0.37);
\path[fill=gray!30] (0,0) to[in = 110, out =190] (-0.71, -0.37) to[in=270, out=270] (0,0);
\node[circle, draw, fill=white, inner sep = 0.4pt,minimum width = 2pt] (a) at (2.4,0.25) {};
\node[circle, draw, fill=white, inner sep = 0.4pt,minimum width = 2pt] (a) at (-0.4,-0.25) {};

\node[circle, draw, fill, inner sep = 0pt,minimum width = 2pt] (-) at (0,0) {};
\node[circle, draw, fill, inner sep = 0pt,minimum width = 2pt] (+) at (2,0) {};
\node[circle, draw, fill, inner sep = 0pt,minimum width = 2pt] (a) at (-0.7, -0.35) {};
\node[circle, draw, fill, inner sep = 0pt,minimum width = 2pt] (b) at (2.7,0.35) {};

\draw[LimeGreen, thick, postaction={decorate},>=stealth, name path=x] (-) to[out = 40, in =140] (+);
\node[LimeGreen, draw=none,fill=none] (x) at (0.5, -0.55) {\tiny$m$};
\draw[LimeGreen, thick, postaction={decorate},>=stealth,name path=xx] (+) to[in = -40, out =220] (-);
\node[LimeGreen, draw=none,fill=none] (x) at (1.6, 0.55) {\tiny$m+1$};
\tikzfillbetween[of=x and xx]{LimeGreen, opacity=0.3};

\draw[DarkOrchid, thick, postaction={decorate},>=stealth, fill=DarkOrchid!30] (+) to[out = -10, in =80] (2.7,-0.35) to[out=260, in =280] (+);
\node[DarkOrchid, draw=none,fill=none] (y) at (-0.3, 0.6) {\tiny$n$};
\draw[DarkOrchid, thick, postaction={decorate},>=stealth,fill=DarkOrchid!30] (-) to[out = 170, in =260] (-0.7,0.35) to[out=80, in =100] (-) ;
\node[DarkOrchid, draw=none,fill=none] (y) at (2.3, -0.6) {\tiny$n$};


\draw[LimeGreen, thick, postaction={decorate},>=stealth, name path=x] (b) to[in = 10, out =280] (+);
\draw[DarkOrchid, thick, postaction={decorate},>=stealth, name path=z] (b) to[out = 100, in =80] (+) ;
\draw[DarkOrchid, thick, postaction={decorate},>=stealth, name path=y] (-) to[out = 190, in =100] (a);
\draw[LimeGreen, thick, postaction={decorate},>=stealth, name path=v] (-) to[in=280, out =260] (a);

\end{scope}
\begin{scope}[shift={(-1,-1.5)}]
\draw[->] (-1,0) -- (0,0);
\end{scope}
\begin{scope}[shift={(-0.25,-1.5)},decoration={
    markings,
    mark=at position 0.55 with {\arrow{>}}}]
\node[circle, draw, fill, inner sep = 0pt,minimum width = 2pt, name path =a] (-) at (0,0) {};
\node[circle, draw, fill, inner sep = 0pt,minimum width = 2pt] (p) at (0.5,0) {};
\node[circle, draw, fill, inner sep = 0pt,minimum width = 2pt] (q) at (1.5,0) {};
\node[circle, draw, fill, inner sep = 0pt,minimum width = 2pt] (+) at (2,0) {};

\draw[LimeGreen, thick, postaction={decorate},>=stealth, name path=z] (-) to (p);
\draw[LimeGreen, thick, postaction={decorate},>=stealth,name path=c] (p) to (q);
\draw[LimeGreen, thick, postaction={decorate},>=stealth,name path=v] (q) to (+);
\node[LimeGreen, draw=none,fill=none] (x) at (1, 0.25) {\tiny$m-1$};

\draw[DarkOrchid, thick, postaction={decorate},>=stealth, name path=x] (-) to[in=120, out=60] (p);
\draw[DarkOrchid, thick, postaction={decorate},>=stealth, name path=xx] (p) to[out=90, in=0] (0.25, 0.5) to[out=180, in=90] (-);
\node[DarkOrchid, draw=none,fill=none] (x) at (0.2, 0.7) {\tiny$n-1$};
\draw[DarkOrchid, thick, postaction={decorate},>=stealth, name path=y] (q) to[in=120, out=60] (+);
\draw[DarkOrchid, thick, postaction={decorate},>=stealth, name path=yy] (+) to[out=90, in=0] (1.75, 0.5) to[out=180, in=90] (q);
\node[DarkOrchid, draw=none,fill=none] (x) at (1.8, 0.7) {\tiny$n-1$};

\draw[LimeGreen, thick, postaction={decorate},>=stealth, name path = b] (+) to[in = -40, out =220] (-);

\tikzfillbetween[of=c and b]{LimeGreen, opacity=0.3};
\tikzfillbetween[of=x and xx]{DarkOrchid, opacity=0.3};
\tikzfillbetween[of=y and yy]{DarkOrchid, opacity=0.3};
\tikzfillbetween[of=x and z]{gray, opacity=0.3};
\tikzfillbetween[of=y and v]{gray, opacity=0.3};
\node[circle, draw, fill=white, inner sep = 0.4pt,minimum width = 2pt] (a) at (0.25,0.08) {};
\node[circle, draw, fill=white, inner sep = 0.4pt,minimum width = 2pt] (a) at (1.75,0.08) {};

\node[LimeGreen, draw=none,fill=none] (x) at (1, -0.6) {\tiny$m$};
\draw[LimeGreen, thick, postaction={decorate},>=stealth, name path = b] (+) to[in = -40, out =220] (-);

\end{scope}
\begin{scope}[shift={(3.5,-1.5)}]
\draw[->] (-1,0) -- (0,0);
\end{scope}
\begin{scope}[shift={(4.75,-2)},decoration={
    markings,
    mark=at position 0.75 with {\arrow{>}}}]
\node[circle, draw, fill, inner sep = 0pt,minimum width = 2pt] (a) at (0,0) {};
\draw[postaction={decorate},DarkOrchid, thick,>=stealth, fill=DarkOrchid!30] (a) to[out = 80, in =180] (0.5,1) to[out=0, in =10] (a);
\draw[postaction={decorate},LimeGreen, thick, >=stealth, fill=LimeGreen!30] (a) to[out = 100, in =0] (-0.5,1) to[out=180, in =170] (a) ;
\draw[LimeGreen,shift={(1.5,1.5)},fill=gray!30, thick] (0:0.5) \foreach \x in {25.714,51.428,...,359} {
            -- (\x:0.5)
        } -- cycle (90:0.5);
\foreach \x in {0,51.428,...,359} {
\draw[DarkOrchid, shift={(1.5,1.5)}, thick] (\x:0.5) -- (\x+25.714:0.5);
        }
\draw[->] (0.8, 1.5) to[out=180, in=90] (0,1);\end{scope}
\end{tikzpicture}
\caption{The maps $\hat X_C\to\hat{\bar X}_C \to \hat X_A$. White nodes are in the $2$-cells whose boundary is mapped $p$-to-$1$.}\label{fig:attaching 2-cells one odd}
\end{figure}
\end{proof}

\begin{lem}\label{lem:von dyck lemma}
The image $\pi(C)$ is isomorphic to to $\mathbb Z/p*\mathbb Z/p$. In particular, $\pi(C) = \pi_1 \hat X_C$. Moreover, $\hat f$ lifts to an embedding in the universal covers.
\end{lem}

\begin{proof}
For simplicity we set $z=xy^{-1}$. 
The image $\pi(C)$ is generated by $z$ and $z'=x^mzx^{-m}$. 
The universal cover of the complex $\hat X_A$ can be identified with the hyperbolic plane. 
Consider the tiling of $\mathbb H^2$ by a triangle with angles $\frac{\pi}{2m+1}, \frac{\pi}{n}, \frac{\pi}{p}$. 
Each vertex of the tiling is a fixed point of a conjugate of one of $x,y,z$, 
and the action of $\hat A$ preserve the type of a vertex (i.e.\ whether it is fixed by a conjugate of $x,y$ or $z$). 
We abuse the notation and identify each vertex $v$ with the conjugate $x^g, y^g$ or $z^g$ 
which generates the stabilizer of $v$ (where $g$ is some element of $\hat A$). 
The tiling is the dual of the universal cover of the complex $\hat X_A$, 
in the way that the vertices of types $x,y,z$ correspond to the $2$-cells with boundary words $x^{2m+1}, y^{n}, (xy^{-1})^p$ respectively.

Consider the Bass-Serre tree $T$ of the free product $\mathbb Z/p*\mathbb Z/p$, i.e.\ a regular tree of valence $p$, where each vertex is stabilized by a conjugate of one of two $\mathbb Z/p$ factors.
In order to prove that $ \pi(C)$ splits as a free product and that $\hat f$ lifts to embedding of the universal covers, we show that there is $\mathbb Z/p*\mathbb Z/p$-equivariant embedding of $T$ in $\mathbb H^2$ where the action of $\mathbb Z/p*\mathbb Z/p$ on $\mathbb H^2$ is the action of the group $\langle z,z'\rangle$.

First consider the union of the orbits of $z$, and of $z'$ under the action of $\pi(C)$. Note that it is a collection of vertices of type $z$. We join $z^g$ and $(z')^g$ by a path consisting of two edges of the tiling meeting at a vertex of type $x$. See Figure~\ref{fig:tiling}.
\begin{figure}
\begin{tikzpicture}
\begin{scope}
\path[draw=black, fill=Black!50] (0,0) -- (-1,0) -- (0,1) --(0,0);
\path[draw=black, fill=Black!50] (1,1) -- (1,0) -- (0,1) --(1,1);
\path[draw=black, fill=Black!50] (1,0) -- (2,0) -- (2,-1) -- (1,0);
\path[draw=black, fill=Black!50] (2,0) -- (3,-1) -- (3,0) -- (2,0);
\path[draw=black, fill=Black!50] (2,0) -- (2,1) -- (3,1) -- (2,0);
\path[draw=black, fill=Black!50] (0,1) -- (-1,1) -- (-1,2) -- (0,1);
\path[draw=black, fill=Black!50] (0,1) -- (0,2) -- (1,2) -- (0,1);
\path[draw=black, fill=Black!50] (-1,0) -- (-2,0) -- (-2,-1) -- (-1,0);
\path[draw=black, fill=Black!50] (-2,-2) -- (-2,-1) -- (-1,-1) -- (-2,-2);
\path[draw=black, fill=Black!50] (-3,-1) -- (-3,-2) -- (-2,-1);

\path[draw=black] (0,0) -- (1,0) -- (2,1);
\path[draw=black] (2,-1) -- (3,-1);
\path[draw=black] (1,1) -- (1,2);
\path[draw=black] (-1,1) -- (-1,0) -- (-1,-1);
\path[draw=black] (-3,-2) -- (-2,-2);
\path[very thick, draw] (3,-1) -- (2,0) -- (1,0) -- (0,1) -- (-1,0) -- (-2,-1) -- (-2,-2);
\path[very thick, draw] (-1,2) -- (0,1) -- (1,2);
\path[very thick, draw] (2,0) -- (3,1);
\path[very thick, draw] (-2,-1) -- (-3,-1);

\draw[YellowOrange] (1.25,0) arc(0:180:0.25);
\node[YellowOrange, draw=none,fill=none] () at (1.2, 0.35) {\tiny $\pi$};

\draw[YellowOrange] (-1.25,0) arc(180:45:0.25);
\node[YellowOrange, draw=none,fill=none] () at (-1.2, 0.35) {\tiny $\pi$};

\draw[YellowOrange] (-0.25,0) arc(180:360:0.25);
\node[YellowOrange, draw=none,fill=none] () at (0, -0.4) {\tiny $\geqslant \pi$};

\draw[YellowOrange] (-0.25,0) arc(180:360:0.25);
\node[YellowOrange, draw=none,fill=none] () at (0, -0.4) {\tiny $\geqslant \pi$};

\draw[YellowOrange] (-1,-0.75) arc(90:-135:0.25);
\node[YellowOrange, draw=none,fill=none] () at (-1, -1.4) {\tiny $\geqslant \pi$};

\draw[YellowOrange] (2.25,-1) arc(0:-225:0.25);
\node[YellowOrange, draw=none,fill=none] () at (2, -1.4) {\tiny $\geqslant \pi$};

\node[circle, draw, fill, inner sep = 0pt,minimum width = 2pt, RoyalBlue, label=above right:\tiny $y_3$] () at (0,0) {};
\node[circle, draw, fill, inner sep = 0pt,minimum width = 2pt, Red] () at (0,1) {};
\node[circle, draw, fill, inner sep = 0pt,minimum width = 2pt, ForestGreen, label=below:\tiny $x_4$] () at (1,0) {};
\node[circle, draw, fill, inner sep = 0pt,minimum width = 2pt, Red, label=above left:\tiny $z_5$] () at (2,0) {};
\node[circle, draw, fill, inner sep = 0pt,minimum width = 2pt, RoyalBlue, label=above right:\tiny $y_5$] () at (2,-1) {};
\node[circle, draw, fill, inner sep = 0pt,minimum width = 2pt, ForestGreen, label=below:\tiny $x_6$] () at (3,-1) {};
\node[circle, draw, fill, inner sep = 0pt,minimum width = 2pt, ForestGreen] () at (-3,-1) {};
\node[circle, draw, fill, inner sep = 0pt,minimum width = 2pt, ForestGreen, label=below right:\tiny $x_2$] () at (-1,0) {};
\node[circle, draw, fill, inner sep = 0pt,minimum width = 2pt, ForestGreen] () at (3,1) {};
\node[circle, draw, fill, inner sep = 0pt,minimum width = 2pt, RoyalBlue, label=above left:\tiny $y_1$] () at (-1,-1) {};
\node[circle, draw, fill, inner sep = 0pt,minimum width = 2pt, RoyalBlue] () at (2,1) {};
\node[circle, draw, fill, inner sep = 0pt,minimum width = 2pt, RoyalBlue] () at (3,0) {};
\node[circle, draw, fill, inner sep = 0pt,minimum width = 2pt, ForestGreen] () at (1,2) {};
\node[circle, draw, fill, inner sep = 0pt,minimum width = 2pt, ForestGreen] () at (-1,2) {};
\node[circle, draw, fill, inner sep = 0pt,minimum width = 2pt, ForestGreen, label=below:\tiny $x_0$] () at (-2,-2) {};
\node[circle, draw, fill, inner sep = 0pt,minimum width = 2pt, RoyalBlue] () at (-1,1) {};
\node[circle, draw, fill, inner sep = 0pt,minimum width = 2pt, RoyalBlue] () at (1,1) {};
\node[circle, draw, fill, inner sep = 0pt,minimum width = 2pt, RoyalBlue] () at (0,2) {};
\node[circle, draw, fill, inner sep = 0pt,minimum width = 2pt, Red, label=above left:\tiny $z_1$] () at (-2,-1) {};
\node[circle, draw, fill, inner sep = 0pt,minimum width = 2pt, RoyalBlue] () at (-2,0) {};
\node[circle, draw, fill, inner sep = 0pt,minimum width = 2pt, RoyalBlue] () at (-3,-2) {};

\node[draw=none,fill=none] () at (-0.36, 1.8) {\tiny$\dots$};
\node[draw=none,fill=none, rotate=-90] () at (2.8, 0.36) {\tiny$\dots$};
\node[draw=none,fill=none, rotate=45] () at (-2.5, -0.5) {\tiny$\dots$};
\node[draw=none,fill=none] () at (1.36, 0.8) {\tiny$\dots$};

\end{scope}

\begin{scope}[shift={(5,0)}]
\draw[fill=Black!50] (0,0) -- (1,0) -- (1,1) -- cycle;
\draw[fill=Black!50] (2,0) -- (3,0) -- (3,1) -- cycle;
\draw[fill=Black!50] (4,0) -- (5,0) -- (5,1) -- cycle;
\draw[fill=Black!50] (1,0) -- (2,0) -- (1.5, -1) -- cycle;
\draw[fill=Black!50] (3,0) -- (4,0) -- (3.5, -1) -- cycle;
\draw[fill=Black!50] (5,0) -- (6,0) -- (5.5, -1) -- cycle;
\draw (0,-1) -- (0,0) -- (0.5, -1) -- (1,0) -- (1,-1);
\draw (2,-1) -- (2,0) -- (2.5, -1) -- (3,0) -- (3,-1);
\draw (4,-1) -- (4,0) -- (4.5, -1) -- (5,0) -- (5,-1);
\draw (6,-1) -- (6,0);
\draw[YellowOrange] (0,-1) -- (6,-1);

\draw[very thick] (0,0) -- (1,1) -- (2,0) -- (3,1) -- (4,0) -- (5,1) -- (6,0);
 
\node[circle, draw, fill, inner sep = 0pt,minimum width = 2pt, ForestGreen, label=above:\tiny $x_0$] () at (0,0) {};
\node[circle, draw, fill, inner sep = 0pt,minimum width = 2pt, label=above right:\tiny $y_1$, RoyalBlue] () at (1,0) {};
\node[circle, draw, fill, inner sep = 0pt,minimum width = 2pt, ForestGreen, label=above:\tiny $x_2$] () at (2,0) {};
\node[circle, draw, fill, inner sep = 0pt,minimum width = 2pt, RoyalBlue, label=above right:\tiny $y_3$] () at (3,0) {};
\node[circle, draw, fill, inner sep = 0pt,minimum width = 2pt, ForestGreen, label=above:\tiny $x_4$] () at (4,0) {};
\node[circle, draw, fill, inner sep = 0pt,minimum width = 2pt, RoyalBlue, label=above right:\tiny $y_5$] () at (5,0) {};
\node[circle, draw, fill, inner sep = 0pt,minimum width = 2pt, ForestGreen, label=above:\tiny $x_6$] () at (6,0) {};
\node[circle, draw, fill, inner sep = 0pt,minimum width = 2pt, Red, label=left:\tiny $z_1$] () at (1,1) {};
\node[circle, draw, fill, inner sep = 0pt,minimum width = 2pt, Red, label=left:\tiny $z_3$] () at (3,1) {};
\node[circle, draw, fill, inner sep = 0pt,minimum width = 2pt, Red, label=left:\tiny $z_5$] () at (5,1) {};
\node[circle, draw, fill, inner sep = 0pt,minimum width = 2pt, Red, label=below:\tiny $s_0$] () at (0.5,-1) {};
\node[circle, draw, fill, inner sep = 0pt,minimum width = 2pt, Red, label=below:\tiny $s_1$] () at (1.5,-1) {};
\node[circle, draw, fill, inner sep = 0pt,minimum width = 2pt, Red, label=below:\tiny $s_2$] () at (2.5,-1) {};
\node[circle, draw, fill, inner sep = 0pt,minimum width = 2pt, Red, label=below:\tiny $s_3$] () at (3.5,-1) {};
\node[circle, draw, fill, inner sep = 0pt,minimum width = 2pt, Red, label=below:\tiny $s_4$] () at (4.5,-1) {};
\node[circle, draw, fill, inner sep = 0pt,minimum width = 2pt, Red, label=below:\tiny $s_5$] () at (5.5,-1) {};
\node[circle, draw, fill, inner sep = 0pt,minimum width = 2pt, Gray, label=below:\tiny $q_1$] () at (1,-1) {};
\node[circle, draw, fill, inner sep = 0pt,minimum width = 2pt, Gray, label=below:\tiny $q_3$] () at (3,-1) {};
\node[circle, draw, fill, inner sep = 0pt,minimum width = 2pt, Gray, label=below:\tiny $q_5$] () at (5,-1) {};
\node[circle, draw, fill, inner sep = 0pt,minimum width = 2pt, Gray, label=below:\tiny $q_0$] () at (0,-1) {};
\node[circle, draw, fill, inner sep = 0pt,minimum width = 2pt, Gray, label=below:\tiny $q_2$] () at (2,-1) {};
\node[circle, draw, fill, inner sep = 0pt,minimum width = 2pt, Gray, label=below:\tiny $q_4$] () at (4,-1) {};
\node[circle, draw, fill, inner sep = 0pt,minimum width = 2pt, Gray, label=below:\tiny $q_6$] () at (6,-1) {};

\node[draw=none,fill=none] () at (1.025, -0.5) {\tiny$\dots$};
\node[draw=none,fill=none] () at (2.025, -0.5) {\tiny$\dots$};
\node[draw=none,fill=none] () at (3.025, -0.5) {\tiny$\dots$};
\node[draw=none,fill=none] () at (4.025, -0.5) {\tiny$\dots$};
\node[draw=none,fill=none] () at (5.025, -0.5) {\tiny$\dots$};

\end{scope}
\end{tikzpicture}
\caption{The vertices $z_1, z_5$ are $\pi(C)$-conjugates of $s$, and the vertex $z_3$ is a $\pi(C)$-conjugate of $z$. Thick edges are in the image of the tree $T$
Odd gray vertices ($q_1,q_3, q_5$) are either green or red, depending on parity of $n$. Even gray vertices  ($q_0,q_2, q_4, q_6$) are either blue or red, depending on parity of $m$. Note that the orange segments are not necessarily edges in the tiling. }
\label{fig:tiling}
\end{figure}
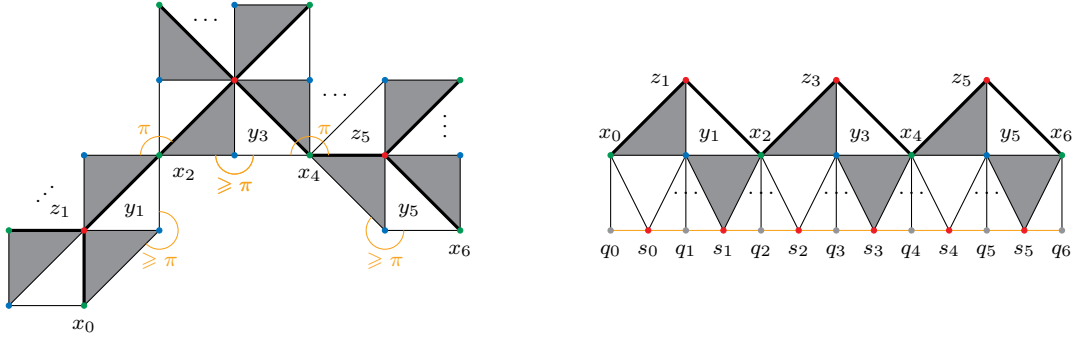

In order to see that $T$ is embedded, we verify that any bi-infinite path $\gamma$ always turning rightmost in the image of $T$ never crosses itself. 
In Figure~\ref{fig:tiling} a part of that path is presented as the path with vertices $x_0, z_1, x_2, z_3, x_4, z_5, x_6$. 
We claim that another path  $\gamma'$ (whose part is labelled by $q_0, s_0, q_1, s_1, \dots, q_5, s_5, q_6$ in Figure~\ref{fig:tiling}) stays in a final distance from $\gamma$, 
and separates $\gamma$ from a subspace of $\mathbb H^2$ which is a union of halfspaces. 
Let us explain how the vertices $q_i, s_i$ are defined. 
The vertex $s_i$ with odd $i$ is the unique vertex other than $z_i$ that forms a triangle with $y_i$ and $x_{i+1}$. 
The vertex $s_i$ with even $i$ is the unique vertex other than $z_{i+1}$ that form a triangle with $x_i$ and $y_{i+1}$. 
In particular, vertices $s_i$ are always of type $z$.
The vertices $q_i$ with odd $i$ are images of a $z_i$ under the $\pi$-rotation at $y_i$, 
i.e.\ the vertex $y_i$ is a midpoint of the segment $[z_i, q_i]$. 
The vertices $q_i$ with even $i$ are chosen so that the angle $\angle s_{i-1}x_iq_i$ and $\angle s_ix_i q_i$ are equal. 
The path $\gamma'$ is obtained by joining each pair $s_{i-1}, q_i$ and$q_i, s_i$ be a geodesic segment (which are not necessarily edges of the tiling).

The vertices $s_{i-1},q_i,s_i$ are not necessarily distinct.
If $n=2$, then $s_{i-1}=q_i=s_i$ for each odd $i$. 
Also, if $2m+1=5$, then $s_{i-1}=q_i=s_i$ for $i=4k+2$.
In that extreme case $\gamma'$ is a geodesic line, as long as $p\geq 6$. 
In more general case, the ``upper'' angle between the segments of $\gamma'$ at each vertex $q_i$ or $s_i$ (i.e.\ the angle of the sector containing $y_i$ or $x_i$ depending on the parity of $i$) is always at most $\pi$. Consequently, the subspace of $\mathbb H^2$ bounded by $\gamma'$ that does not contain the path $\gamma$ is a union of halfspaces. This proves our claim.
\begin{com}
Add more details
\end{com}
\end{proof}

Our approach in the last two theorems fails in the case where $M=2m+1,N=2n+1$. Indeed, the fiber product $\bar X_C\otimes_{X_A}\bar X_C$ is too ``large''. The fiber product was computed in \Cref{rem:both odd}. After attaching $2$-cells along $x^{2m+1}, y^{2n+1}$ and $(x^{-1}y)^p$, the resulting $2$-complex has fundamental group $\mathbb Z*\mathbb Z/p$. 

\subsection{Summary of  residual finiteness of three generator Artin groups}
To summarize, the only three generator Artin groups that are not known to be residually finite are $\text{Art}_{33(2m+1)}$ for $m\geq 2$, $\text{Art}_{2(2m+1)(2n+1)}$ for $m+n\geq 4$ and $\text{Art}_{23(2m)}$ for $m\geq 4$. Indeed, if at least one label is $\infty$ then the defining graph is a tree, and hence virtually special \cite{Brunner92}, \cite{HermillerMeier99}, \cite{LiuGraphManifolds}, \cite{PrzytyckiWiseGraphManifolds}. 
Artin groups $\text{Art}_{22M}$ for any $M\geq 2$, and $\text{Art}_{23M}$ where $M\in\{3,4,5\}$ are spherical, and so linear \cite{CohenWales2002}, \cite{Digne2003}. The cases $(3,3,3), (2,4,4)$ and $(2,3,6)$ follows from \cite{Squier87}. The cases where $M,N,P\geq 3$, except the case of $(3,3,2m+1)$, were covered by \cite{JankiewiczArtinRf}. 

\bibliographystyle{alpha}
\bibliography{../../../kasia}

\end{document}